\documentclass[12pt,reqno]{amsart}
\usepackage{amsmath, amsthm, amssymb}
\usepackage{mathtools}
\usepackage{mathrsfs}  
\usepackage{stmaryrd}
\usepackage{comment}
\usepackage{marginnote}

\usepackage[margin=1in]{geometry}
\usepackage{parskip}
\usepackage[shortlabels]{enumitem}

\usepackage{xcolor}

\usepackage[bookmarks,colorlinks,breaklinks]{hyperref}
\hypersetup{linkcolor=red,citecolor=blue,
  filecolor=dullmagenta,urlcolor=blue}

\usepackage[numbers]{natbib}
\usepackage[capitalise]{cleveref}

\usepackage{graphicx}
\usepackage{tikz}
\usetikzlibrary{matrix, arrows}
\usetikzlibrary{cd}
\usepackage[all]{xy}
   \SelectTips{cm}{10}

\usepackage{hyperref}
\usepackage{enumitem}
\usepackage{graphicx}

\makeatletter
\def\thm@space@setup{%
  \thm@preskip=0.5em\thm@postskip=\thm@preskip%
}
\makeatother

\newtheoremstyle{named}{}{}{\\itshape}{}{\bfseries}{.}{.5em}{\thmnote{#3's }#1}
\theoremstyle{named}

\theoremstyle{plain}
\newtheorem{thm}{Theorem}[section]

\newtheorem{prop}[thm]{Proposition}
\newtheorem{lem}[thm]{Lemma}

\newtheorem{cor}[thm]{Corollary}
\theoremstyle{definition}
\newtheorem{defn}[thm]{Definition}

\newtheorem{assumption}[thm]{Assumption}
\newtheorem{convention}[thm]{Convention}

\theoremstyle{remark}
\newtheorem{rmk}[thm]{Remark}
\crefformat{footnote}{#2\footnotemark[#1]#3}

\newcommand{\Hom}{\mathrm{Hom}}

\newcommand{\A}{\mathbb{A}}
\newcommand{\CC}{\mathbb{C}}
\newcommand{\QQ}{\mathbb{Q}}
\newcommand{\RR}{\mathbb{R}}

\newcommand{\ZZ}{\mathbb{Z}}

\newcommand{\Qlb}{\overline{\mathbb{Q}}_{\ell}}

\newcommand{\GL}{\mathrm{GL}}

\newcommand{\Gal}{\mathrm{Gal}}
\newcommand{\mf}{\mathfrak}
\newcommand{\mc}[1]{\mathcal{#1}}

\newcommand{\mr}[1]{\mathrm{#1}}
\newcommand{\mbb}[1]{\mathbb{#1}}
\newcommand{\ol}[1]{\overline{#1}}
\newcommand{\ul}[1]{\underline{#1}}

\newcommand{\ad}{\mathrm{ad}}

\newcommand{\Spec}{\mathrm{Spec}}

\newcommand{\Sh}{\mathrm{Sh}}
\newcommand{\ltolp}{\lambda \rightsquigarrow \lambda'} 

\newcommand{\fgder}{\mathfrak{g}^{\mathrm{der}}}

\newcommand{\Frob}{\mathrm{Frob}}

\DeclareMathOperator{\Aut}{Aut}
\DeclareMathOperator{\Out}{Out}
\DeclareMathOperator{\Res}{Res}

\DeclareMathOperator{\inn}{inn} 
\DeclareMathOperator{\rec}{rec}

\newcommand{\git}{{\,\!\sslash\!\,}}

\usepackage[normalem]{ulem}
\usepackage[utf8]{inputenc}

\title{Compatibility of canonical $\ell$-adic local systems on adjoint Shimura varieties}

\thanks{We thank Jacob Tsimerman for discussing some of this material. We are grateful to the anonymous refereee for a very helpful report. S.P. was supported by NSF grant DMS-2120325 and DMS-1752313. C.K. was supported by Samsung Science
and Technology Foundation under Project Number SSTF-BA2001-02}

\begin{document}

\author[C.~Klevdal]{Christian Klevdal}
\address{Department of Mathematics, University of California San Diego\\ 9500 Gilman Drive\\ La Jolla, CA 92093, USA}
\email{cklevdal@ucsd.edu}
\author[S.~Patrikis]{Stefan Patrikis}
\address{Department of Mathematics, The Ohio State University\\ 100 Math Tower\\ 231 West 18th Avenue\\ Columbus, OH 43210, USA}
\email{patrikis.1@osu.edu}

\begin{abstract}
    For a Shimura variety $(G, X)$ in the superrigid regime and neat level subgroup $K_0$, we show that the canonical family of $\ell$-adic representations associated to a number field point $y \in \Sh_{K_0}(G, X)(F)$, 
    \[
    \left\{ \rho_{y, \ell} \colon \mr{Gal}(\ol{\QQ}/F) \to G^{\mr{ad}}(\QQ_{\ell}) \right\}_{\ell},
    \]
    form a compatible system of $G^{\mr{ad}}(\QQ_{\ell})$-representations: there is an integer $N(y)$ such that for all $\ell$, $\rho_{y, \ell}$ is unramified away from $N(y) \ell$, and for all $\ell \neq \ell'$ and $v \nmid N(y)\ell \ell'$, the semisimple parts of the conjugacy classes of $\rho_{y, \ell}(\mr{Frob}_v)$ and $\rho_{y, \ell'}(\mr{Frob}_v)$ are ($\QQ$-rational and) equal. We deduce this from a stronger compatibility result for the canonical $G(\QQ_{\ell})$-valued local systems on connected Shimura varieties inside $\mr{Sh}_{K_0}(G, X)$. Our theorems apply in particular to Shimura varieties of non-abelian type and represent the first such independence-of-$\ell$ results in non-abelian type.
\end{abstract}

\maketitle

\section{Introduction}
Let $X$ be a smooth quasi-projective variety over $\CC$, let $G$ be a connected reductive group, and consider a representation
\[
\rho \colon \Gamma := \pi_1^{\mr{top}}(X(\CC), x) \to G(\CC).
\]
It is natural to ask for conditions that would ensure $\rho$ arises from the cohomology of a family of varieties over $X$, for instance in the following sense: for each representation $r \colon G \to \mr{GL}_n$ there is a smooth projective family $f \colon W \to U$ over a (Zariski) open subset $U \subset X$ such that $r \circ \rho|_{\pi_1(U)}$ is isomorphic to a subquotient of the local systems $\oplus_s R^sf_* \CC$. A conjecture of Carlos Simpson predicts a set of sufficient conditions for such a $\rho$ to come from cohomology in this manner.  Let us assume for simplicity that $\rho(\Gamma) \cap G^{\mr{der}}(\CC)$ is Zariski-dense in $G^{\mr{der}}(\CC)$. Suppose that $\rho$ satisfies:
\begin{itemize}
\item the image of $\mr{ab}(\rho)$, $\mr{ab}$ the projection $G \xrightarrow{\mr{ab}} G/G^{\mr{der}}$, is finite;
\item $\rho$ has quasi-unipotent local monodromies along the boundary components $D_i$ of a strict normal crossings divisor compactification $\ol{X} \supset X$, with $\ol{X} \setminus X= \cup_{i=1}^a D_i$;
\item and $\rho$ is $G$-rigid in the sense that it defines an isolated point of the moduli space of representations $\rho' \colon \pi_1(X, x) \to G(\CC)$ such that $\det(\rho')=\det(\rho)$ and, writing $\gamma_i$ for a generator in $\pi_1(X, x)$ of the local monodromy along $D_i$, $\rho'(\gamma_i)$ is conjugate to $\rho(\gamma_i)$.\footnote{Replacing this last conjugacy condition by the requirement that $\rho'(\gamma_i)$ lies in a fixed closed union of conjugacy classes containing $\rho(\gamma_i)$, this moduli space can be constructed using GIT; for a more general formulation, see \cite[\S 4]{klevdal-patrikis:G-rigid}, with the GIT construction appearing in Remark 4.10 of \textit{loc. cit.}}
\end{itemize}
Then Simpson conjectures $\rho$ comes from cohomology. In place of the rigidity condition, it is often easier to work with the condition of cohomological rigidity, that
\[\
\ker(H^1(\pi_1(X, x), \rho(\fgder)) \to \bigoplus_{i=1}^a H^1(\langle \gamma_i \rangle, \rho(\fgder))=0.
\] 
Here we write $\rho(\fgder)$ for the composition of $\rho$ with the adjoint action on the Lie algebra $\fgder$ of $G^{\mr{der}}$. This kernel is a tangent space, and cohomological rigidity amounts to saying that $\rho$ is a smooth isolated point of the moduli space.

A particularly strong form of the cohomological rigidity condition arises from Margulis's superrigidity theorems, and in this paper we present a new application of these results to the arithmetic of Shimura varieties of non-abelian type. Suppose $(G, X)$ is a Shimura datum, and assume for simplicity that $G^{\mr{ad}}$ is $\QQ$-simple and that $Z_G(\QQ) \subset Z_G(\A_f)$ is discrete. Finally, and crucially, assume that $\mr{rk}_{\RR}(G^{\mr{ad}}_{\RR}) \geq 2$. Let $\Gamma \subset G^{\mr{ad}}(\QQ)$ be an arithmetic subgroup, which for us will arise as the fundamental group of a connected Shimura variety $\Gamma \backslash X^+$. Then Margulis has proven (\cite[]{margulis:discrete}) that for any rational representation $G \to \mr{GL}(V)$, the group $H^1(\Gamma, V)$ vanishes; taking in particular $V= \rho(\fgder)$, for $\rho \colon \Gamma \to G^{\mr{ad}}(\QQ)$ the inclusion, the representation $\rho$ is (strongly) $G$-cohomologically rigid. Taking $V= \QQ$ to be the trivial representation, we see that $\mr{ab}(\rho)$ has finite order, and one can further check (see \cite[\S 6.1]{lan-suh:vanishingncpel}) that $\rho$ has quasi-unipotent (even unipotent) local monodromies along the boundary components of a suitable toroidal compactification.  Simpson's conjecture would thus imply that $\rho$ comes from cohomology, and in this special case the conjecture is in essence Deligne's \textit{r\^{e}ve} that general Shimura varieties (with rational weight homomorphism) are moduli of motives.

We do not solve this problem. In the absence of the motives themselves, we study (what should be) their $\ell$-adic realizations and prove that they form compatible systems of $\ell$-adic Galois representations. To make this precise, we state a weak form of our main result:
\begin{thm}[See Corollary \ref{specialthm}]\label{specialthmintro}
Continue with the above assumptions: $(G, X)$ is a Shimura datum with $G^{\mr{ad}}$ $\QQ$-simple, $Z_G(\QQ) \subset Z_G(\A_f)$ discrete, and $\mr{rk}_{\RR}(G_\RR^{\mr{ad}}) \geq 2$. Let $K_0 \subset G(\A_f)$ be a neat compact open subgroup, and let $\mr{Sh}_{K_0}$ be the canonical model over the reflex field $E(G, X)$ of the Shimura variety $\mr{Sh}_{K_0}(G, X)(\CC)=G(\QQ) \backslash (X \times G(\A_f))/K_0$.

Let $F$ be a number field, and let $y \in \mr{Sh}_{K_0}(F)$ be an $F$-rational point. For each $\ell$, there is canonically associated with $y$ a representation $\rho_{y, \ell} \colon \mr{Gal}(\ol{\QQ}/F) \to G(\QQ_{\ell})$, and the projections $\rho_{y, \ell}^{\mr{ad}} \colon \mr{Gal}(\ol{\QQ}/F) \to G^{\mr{ad}}(\QQ_{\ell})$ form a compatible sense in the following sense: for each $y$, there exists $N(y) \in \ZZ$ such that for all $\ell$ and finite places $v \nmid N(y) \ell$ of $F$, 
\begin{itemize} 
\item $\rho^{\ad}_{y, \ell}$ is unramified at $v$;
\item and the semisimple conjugacy classes underlying $\rho^{\ad}_{y, \ell}(\mr{Frob}_v)$ lie in $[G^{\mr{ad}} \git G^{\mr{ad}}](\QQ)$ and are independent of $\ell$.
\end{itemize}
\end{thm}
\begin{rmk}\label{weakmainrmk}
\begin{enumerate}
\item Here $[G^{\mr{ad}} \git G^{\mr{ad}}]$ denotes the GIT quotient, with $G^{\mr{ad}}$ acting on itself by conjugation; for an algebraically closed field $k$, the map $G^{\mr{ad}}(k) \to [G^{\mr{ad}} \git G^{\mr{ad}}](k)$ induces a bijection between the target and the set of semisimple conjugacy classes in $G^{\mr{ad}}(k)$. Note that the conclusion of the theorem does not assert that $\rho_{y, \ell}^{\mr{ad}}(\mr{Frob}_v)$ is conjugate to an element of $G^{\mr{ad}}(\QQ)$.
\item The need to project to $G^{\mr{ad}}$ to establish the independence of $\ell$ is a limitation of our argument; we will remark more on this point after stating Theorem \ref{mainthmintro}, when we summarize the proof. 
\item The Shimura varieties of non-abelian type all satisfy the adjoint real rank condition of the theorem. They come in four classes: type $E_6$; type $E_7$; ``mixed" type $D_n$ ($G^{\mr{ad}}$ is a, for simplicity, $\QQ$-simple group such that the Lie algebra of $G^{\mr{ad}}_{\RR}$ has simple factors of both types $\mf{so}(2, 2n-2)$, denoted type $D_n^{\RR}$ in the literature, and $\mf{so}^*(2n)$, denoted $D_n^{\mathbb{H}}$); and the case where $G^{\mr{ad}}_{\RR}$ is purely of type $D_n^{\mathbb{H}}$ but $G^{\mr{der}}$ is simply-connected. Theorem \ref{specialthmintro} gives the first such $\ell$-adic compatibility results for any choice of $G$ in the first three cases; in the fourth case, our theorem gives new results when $G^{\mr{der}}$ is simply-connected or (in type $D_{2n}$) when it is a double-cover of $G^{\mr{ad}}$ that is not a form of $\mr{SO}_{4n}$ (the point being that, even though we only prove compatibility in $G^{\mr{ad}}$, the level structure $K_0$ can induce an arithmetic but non-congruence subgroup of $G^{\mr{ad}}(\QQ)$ that in the notation of \cite[Definition 2.1.6]{deligne:canonicalmodels} is not open for the topology $\tau(G_1)$ for any $G_1$ isogenous to $G^{\mr{ad}}$ that yields a connected Shimura variety of abelian type: see \cite[Th\'eor\`eme 2.7.20]{deligne:canonicalmodels}). 
\item When $(G, X)=(\mr{GSp}_{2n}, \mathcal{H}_{n}^{\pm})$ is the Siegel Shimura datum (which for $n>1$ satisfies the hypotheses of the theorem), the representations $\rho_{y, \ell}$ are simply the $\ell$-adic Tate modules of the abelian varieties underlying the points $y$, in their moduli interpretation. More generally, for Shimura varieties of Hodge type, stronger results than Theorem \ref{specialthmintro} are known. We cite in particular the recent work of Kisin-Zhou (\cite{kisin-zhou}; but see also \cite{commelin:abelian}, \cite{kisin:modpabelian}), which proves that for any abelian variety $A$ over $F$ and prime $v$ of good reduction (and $v \nmid 2$) of $A$, the semisimple conjugacy classes $\rho_{A, \ell}(\mr{Frob}_v)$ are independent of $\ell$ in the Mumford-Tate group of $A$ (possibly after replacing $F$ by a finite extension). They prove this by constructing canonical integral models, equipped with Kottwitz-Rapoport stratification on the special fiber, for certain Shimura varieties of Hodge type, notably allowing the group $G_{\QQ_p}$ ($p$ being the rational prime below some $v$ as in Theorem \ref{specialthmintro}) to be merely quasi-split at $p$ and the level to be a very special parahoric. This latter condition implies the ordinary locus is dense, and Kisin-Zhou prove a Serre-Tate style lifting result that points in the ordinary locus of the special fiber lift to special (CM) points. This yields the compatibility on the dense open ordinary locus, and by careful choice of curves (Bertini theorems) passing through both the ordinary locus and points in successive strata of the Kottwitz-Rapoport stratification, the function field Langlands correspondence then yields the independence-of-$\ell$ property in the entire special fiber. Such methods of course rely heavily on the deep theory of canonical integral models, developed over decades, of Shimura varieties of abelian type; in the generality of our main theorems, we have no access to such integral model technology.  
\item The representations $\rho_{y, \ell}$ of Theorem \ref{specialthmintro} are known to be de Rham by the de Rham rigidity theorem of Liu-Zhu (\cite{liu-zhu:dRrigidity}), which reduces the de Rham property for any point of the Shimura variety to the corresponding property for a single point (of each geometrically connected component); for special points this is easy. Establishing compatibility at $v \vert \ell$ as well (for $v$ that are suitably primes of good reduction) is work in progress with Ohio State PhD student Jake Huryn and Kiran Kedlaya. See also Remark \ref{mainthmintrormk} for the known crystalline properties of these Galois representations.  
\item In particular, we obtain a rich source of $E_6$ and $E_7$ compatible systems of Galois representations over number fields, and applying the Hilbert irreducibility theorem (to Theorem \ref{mainthmintro}) we can ensure that many of these have Zariski-dense image, at least for some $\ell$. In each case there is one previous construction of compatible systems of $E_6(\ol{\QQ}_{\ell})$ or $E_7(\ol{\QQ}_{\ell})$-valued Galois representations, namely \cite{bcelmpp} (for $E_6$, using Galois deformation theory and potential automorphy theorems) and \cite{yun:exceptional} (for $E_7$, using the construction of rigid Hecke eigensheaves on $\mathbb{P}^1 \setminus \{0, 1, \infty\}$). Our representations are distinct (for Hodge-theoretic reasons) from those of \cite{bcelmpp} and \cite{yun:exceptional} and come in higher-dimensional families, but in our case we cannot control the possible number fields $F$. 
\end{enumerate}
\end{rmk}
Theorem \ref{specialthmintro} follows from specializing our main theorem, which establishes compatibility properties for the canonical $\ell$-adic local systems on the (connected) Shimura varieties. Consider a Shimura datum $(G, X)$ and neat level structure $K_0$ as in Theorem \ref{specialthmintro}. For each compact open subgroup $K \subset K_0$ we have a morphism $\mr{Sh}_K \to \mr{Sh}_{K_0}$ of canonical models over $E(G, X)$, and we let $\mr{Sh}= \varprojlim_K \mr{Sh}_K$ be the limit; then $\mr{Sh}(\CC)= G(\QQ) \backslash (X \times G(\A_f))= G(\QQ)_+ \backslash (X^+ \times G(\A_f))$ for a choice $X^+$ of connected component of $X$. Fix a basepoint $s \in \mr{Sh}(\CC)$, inducing basepoints $sK$ at all finite levels, and let $S_{K_0, s}$ be the geometrically connected, over some finite extension $E_{K_0, s}$ of $E(G, X)$, component of $\mr{Sh}_{K_0}$ containing $sK_0$. From the choice of basepoint one canonically obtains a homomorphism $\pi_1(S_{K_0, s}, s) \to K_0$, inducing (by projection) for each $\ell$ representations 
\[
\rho_{K_0, s, \ell} \colon \pi_1(S_{K_0, s}, s) \to G(\QQ_{\ell}).
\]
(See \S \ref{localsysdef} for the details of this construction.) These are the canonical $\ell$-adic local systems whose compatibility we study, and our main theorem, which readily implies Theorem \ref{specialthmintro}, is:

\begin{thm}[See Theorem \ref{mainthm}]\label{mainthmintro}
Continue with the hypotheses on $(G, X)$ and $K_0$ of Theorem \ref{specialthmintro}. Fix a basepoint $s \in \mr{Sh}(\CC)$, and let $\rho_{\ell} \colon \pi_1(S_{K_0, s}, s) \to G^{\mr{ad}}(\QQ_{\ell})$ denote the adjoint projection of the homomorphism $\rho_{K_0, s, \ell}$. 

Then there is an integer $N$ and an integral model $\mc{S}_{K_0, s} \to \Spec(\mc{O}_{E_{K_0, s}}[1/N])$ such that:
\begin{enumerate}
\item For all $\ell$, $\rho_{\ell}$ extends to (i.e., factors through) a homomorphism 
\[
\pi_1(\mc{S}_{K_0, s}[1/\ell], s) \to G^{\mr{ad}}(\QQ_{\ell}).
\]
\item For all $\ell$ and $\ell'$, and all closed points $x \in \mc{S}_{K_0, s}[1/(\ell \ell')]$, the associated semisimple conjugacy classes of $\rho_{\ell}(\mr{Frob}_x)$ and $\rho_{\ell'}(\mr{Frob}_x)$ (which are elements of $[G^{\mr{ad}} \git G^{\mr{ad}}](\QQ_{\ell})$ and $[G^{\mr{ad}} \git G^{\mr{ad}}](\QQ_{\ell'})$, respectively) are $\QQ$-rational and equal to the same element of $[G^{\mr{ad}} \git G^{\mr{ad}}](\QQ)$.
\end{enumerate}
\end{thm}

\begin{rmk}\label{mainthmintrormk}
\begin{enumerate}
\item For a number field $F$, an element $y \in \mr{Sh}_{K_0}(F)$ is induced by an element $y \in \mc{S}_{K_0, s}(\mc{O}_F[1/N(y)])$ for some integer $N(y)$ (divisible by $N$). Pulling back the compatibility in Theorem \ref{mainthmintro} to $\Spec(\mc{O}_F[1/N(y)])$ we obtain Theorem \ref{specialthmintro}.
\item The construction of the integral model is ``soft," by spreading out a smooth compactification of $S_{K_0, s}$ whose boundary is a strict normal crossings divisor. This suffices for arguments with tame specialization maps (as in \S \ref{abstract}) and for arguments with special points, since in \cite[Theorem 4.1]{pila-shankar-tsimerman:andre-oort} it is shown that all special points are $\ol{\ZZ}[1/N]$-integral (for some $N$ independent of the special point) for such a model. 
\item By \cite[Theorem 5.1]{pila-shankar-tsimerman:andre-oort} (essentially relying on the Appendix to \textit{loc. cit.} by Esnault and Groechenig), there is a number field $E'$ such that for almost all $\ell$ these local systems become crystalline when restricted to $E'_v$ for $v \vert \ell$.
\end{enumerate}
\end{rmk}

\begin{rmk}\label{langlands-rapoport-remark}
After the initial submission of this paper in March 2023, Bakker, Shankar, and Tsimerman posted a remarkable paper (\cite{bakker-shankar-tsimerman:canonicalmodels}) that establishes existence of canonical integral models, away from finitely many primes, for Shimura varieties of non-abelian type, as well as a number of applications. Their models are the ``soft" integral models that we also use, following \cite{pila-shankar-tsimerman:andre-oort}, and have an inexplicit set of bad primes as in our Theorem \ref{mainthmintro}. Some of their results depend on the results of the present paper. With the extension property of integral models established, we are naturally led to wonder what approximations to the Langlands-Rapoport conjecture describing mod $p$ points group-theoretically (see \cite[\S 3]{kisin:modpabelian} for background; a proof in abelian type is the main theorem of \textit{loc. cit.}) might be in reach. Very roughly, the conjecture gives a group-theoretic decomposition
$\Sh_{K_p}(\ol{\mathbb{F}}_p)= \bigsqcup_{[\phi]} S(\phi)$
of the mod $p$ points (at infinite level away from $p$ and hyperspecial level $K_p$ at $p$) into sets $S(\phi)$ of a group-theoretic nature attached to a ``fake" $G$-motive over $\ol{\mathbb{F}}_p$, rigorously described as a suitable representation of the quasi-motivic Galois gerbe (see \cite[\S 3.1]{kisin:modpabelian}). An intermediate step in attacks on the conjecture is the analysis of Kottwitz triples (see \cite[\S 4.3]{kisin:modpabelian}), which can be associated to the parameters $\phi$ or, as in Kottwitz's work (\cite{kottwitz:lambda-adic}, \cite{kottwitz:annarbor}, \cite{kottwitz:modppoints}), used directly to tackle applications to the construction of automorphic Galois representations. Theorem \ref{mainthmintro} is a first step in associating Kottwitz triples to elements of $\Sh_{K_p}(\ol{\mathbb{F}}_p)$: at least when $G= G^{\mr{ad}}$, it produces the desired $(\gamma_{\ell})_{\ell \neq p} \in G(\mathbb{A}^p_f)$ and produces a weak version of the $\gamma_0 \in G(\QQ)$ that must have the same $G(\ol{\QQ}_{\ell})$-conjugacy class as each $\gamma_{\ell}$ (we only obtain a $\QQ$-rational conjugacy class so still face a stable conjugacy obstruction). We expect our forthcoming work noted in part (5) of Remark \ref{weakmainrmk} to produce Kottwitz's $\gamma_p$, and we plan to continue the analysis of this problem in the future. A closely-related problem (as evident in the arguments of \cite{kisin:modpabelian}) is the lifting of mod $p$ points to special points in characteristic zero: for the $\mu$-ordinary locus, \cite[Theorem 9.7]{bakker-shankar-tsimerman:canonicalmodels} establishes this liftability using Theorem \ref{mainthmintro}. 
\end{rmk}

We now sketch the path to Theorem \ref{mainthmintro}. In the first part of the theorem, each $\rho_{\ell}$ has good reduction away from finitely many bad fibers based on general ramification arguments (using that the representations are geometrically irreducible). To show that the bad fibers can be taken to be independent of $\ell$ we use the companions correspondence (Deligne's conjecture and Drinfeld's theorem) for smooth varieties over finite fields: in brief, if $\rho_{\ell}$ is good along the $v$-fiber $\mc{S}_v$ of $\mc{S}_{K_0, s}$ and some $\rho_{\ell'}$ is not (for $v \nmid N \ell \ell'$), then we consider the $\ell'$-adic companion $\rho_{v, \ell \rightsquigarrow \ell'}$ of $\rho_{\ell}|_{\mc{S}_v}$. It is tamely ramified along the boundary divisor and can be pulled back along the tame specialization map to a representation of $\pi_1(S_{K_0, s} \otimes \CC, s)$ and hence of the topological fundamental group $\Gamma= \pi_1^{\mr{top}}(S_{K_0, s}(\CC), s)$. This latter group is an arithmetic group in a higher-rank Lie group, so Margulis's superrigidity theorems apply. The precise form we use relies on our target $G^{\mr{ad}}$ being an adjoint group but ensures the strong conclusion that any two homomorphisms $\rho_1, \rho_2 \colon \Gamma \to G^{\mr{ad}}(\Omega)$ with Zariski-dense image, and with $\Omega$ any algebraically closed field of characteristic zero, satisfy $\rho_1=\tau(\rho_2)$ for some automorphism $\tau$ of $G^{\mr{ad}}_{\Omega}$. (See Definition \ref{superrigdef}, where we abstract this notion, and Proposition \ref{svsuperrig}, where we verify that results of Margulis imply it holds in our setting.) At this point we can deduce (see Proposition \ref{nearcompatible}) that the sets of bad fibers are independent of $\ell$ and that for all good fibers $v$, $\rho_{\ell'}|_{\mc{S}_v}= \tau(\rho_{v, \ell \rightsquigarrow \ell'})$, i.e., that along $\mc{S}_v$ $\rho_{\ell}$ and $\rho_{\ell'}$ are at least companions up to the automorphism $\tau$, which may depend on $v$, $\ell$, and $\ell'$, of $G^{\mr{ad}}_{\ol{\QQ}_{\ell'}}$. This part of the argument is presented in an abstract form in \S \ref{abstract}.\footnote{Even though it is possible to prove more general statements (see Remark \ref{non-adjoint}) for some of this section's results, at the referee's suggestion in this abstract discussion we restrict to local systems valued in adjoint groups: ultimately we cannot prove our main theorem for the non-adjoint Galois representations, and the arguments of \S \ref{abstract} then avoid unenlightening complications.} We note that these arguments are guided by previous work on Simpson's integrality conjecture for (cohomologically) rigid representations, by Esnault-Groechenig (\cite{esnault-groechenig:rigid}) in the case $G= \mr{GL}_n$ and subsequently by us (\cite{klevdal-patrikis:G-rigid}) for general $G$.

We specialize the results of \S \ref{abstract} to the case of rigid Shimura varieties at the start of \S \ref{svsection}, and we then face the problem of showing these automorphisms $\tau$ are inner. Note that the group $G^{\mr{ad}}$ will be some restriction of scalars $\Res_{F/\QQ}(G^{s})$ for $G^s/F$ an absolutely simple group (and, in the Shimura variety setting, $F$ totally real). Thus over an algebraically closed $\Omega$, $G^{\mr{ad}}_{\Omega} \cong \prod_{\sigma \colon F \to \Omega} G^s_{\sigma}$ will have many outer automorphisms coming from permutations, in addition to any Dynkin diagram automorphisms of the simple factors. If we knew Theorem \ref{specialthmintro} in advance---compatibility of the local systems after specialization to ``horizontal" sections (in contrast to the vertical fibers $\mc{S}_v$)---then we could deduce Theorem \ref{mainthmintro}; this is indeed shown in great generality in \cite{drinfeld:deligneconj}. We understand in advance only the horizontal sections coming from special points of $\mr{Sh}_{K_0}$, and we use judicious choice of special points to show all such $\tau$ are inner. The rough idea is, with $v$ above some rational prime $p$ fixed, to choose a special point $y \colon \mc{O}_C[1/N(y)] \to \mc{S}_{K_0, s}$, for some number field $C \subset \mbb{C}$ and integer $N(y)$, such that some Frobenius element over $v$ is in very general position. More precisely, we arrange that:
\begin{itemize}
\item $C$ is unramified over $p$, and $p \nmid N(y)$;
\item the Mumford-Tate group $T_y \subset G$ of $y$ is a torus whose image $T_y^{\mr{ad}}$ in $G^{\mr{ad}}$ is irreducible (contains no proper $\QQ$-subtori);
\item for $w$ a place of $C$ above $v$, the specialization $\rho_{\ell, y}(\Frob_w)$ has infinite order.
\end{itemize}
As explained in the proof of Theorem \ref{mainthm}, these properties suffice along with the known compatibility of $\rho_{\ell, y}$ and $\rho_{\ell', y}$ to force $\tau$ to be inner. Producing such special points relies on four ingredients: 
\begin{itemize}
\item elementary calculation with the Galois action on special points (which characterizes canonical models); 
\item the fact that the models we use have a uniform integrality property for all special points, as established in \cite[Theorem 4.1]{pila-shankar-tsimerman:andre-oort};
\item a generalization of a result (\cite[Theorem 1]{prasad-rapinchuk:irrtori}) of Prasad and Rapinchuk (relying on weak approximation for and rationality of the variety of maximal tori) to produce many irreducible tori $T_y$ with desired local properties;
\item an application of work of Kisin-Madapusi Pera-Shin (\cite{kisin-madapusipera-shin:hondatate}) to arrange a positivity property that ensures $\rho_{\ell, y}(\Frob_w)$ has infinite order.
\end{itemize}
The details of these arguments appear in \S \ref{svsection}.
\subsection{Notation}
For a point $x$ of a scheme $X$, we write $\kappa(x)$ for the residue field at $x$. For a scheme $S$ and two $S$-schemes $X \to S$ and $T \to S$ we will write $X_T$ for the base-change $X \times_S T$. When $T=\Spec(R)$ is affine, we will simply write $X_R$ for $X_T$. When $S= \Spec(A)$ is also affine, we sometimes write $X \otimes_A R$ for this base-change. 

For a scheme $X$ and a geometric point $\bar{s}$, we write $\pi_1(X, \bar{s})$ for the \'etale fundamental group. When $X$ is a scheme of finite-type over $\ZZ$, $x \in X$ is a closed point, we write $\mr{Frob}_x$ for any representative in $\pi_1(X, \bar{s})$ of the conjugacy class of geometric Frobenius elements at $x$. When $X$ is moreover regular and admits a regular compactification with boundary a normal crossings divisor $D$, we write $\pi_1^t(X, \bar{s})$ for the tame (along $D$) fundamental group of $X$ (\cite{grothendieck:tame}). See \cite[Theorem 4.4]{kerz-schmidt:tameness} for the independence of choice of compactification and several equivalent notions of tameness.

We write $\mathbb{A}$ for the ring of adeles of $\QQ$ and $\mathbb{A}_f$ for the finite adeles. For $F$ a field and $F^{\mr{sep}}$ a separable algebraic closure we write $\Gamma_F$ for $\mr{Gal}(F^{\mr{sep}}/F)$ when $F$ is already given as a subfield of $\CC$, we always take $F^{\mr{sep}} = \ol{F}$ to be the algebraic closure of $F$ in $\CC$. For a number field $F$ we write $\mr{rec}_F \colon \mathbb{A}_F^\times \to \Gamma_F^{\mr{ab}}$ for the reciprocity homomorphism of global class field theory, normalized to take uniformizers to geometric Frobenii.

Throughout $G$ will be a connected reductive group over $\QQ$. We write $G^{\mr{ad}}$ for the adjoint group and $G^{\mr{der}}$ for the derived group; for a maximal torus $T$ in $G$ we will also abusively write $T^{\mr{ad}}$ and $T^{\mr{der}}$ for the image of $T$ in $G^{\mr{ad}}$ and $T \cap G^{\mr{der}}$. We write $G(\RR)^+$ for the connected component of the identity and $G(\QQ)_+$ for the subgroup of $G(\QQ)$ projecting to $G^{\mr{ad}}(\RR)^+$. For $L/K$ a finite separable extension of fields and $H$ an affine algebraic group over $L$, we write $\Res_{L/K}(H)$ for the Weil restriction to $K$ of $H$. For an algebraic group $H$ over $K$ and $h \in H(K)$, we write $\inn(h)$ for the conjugation by $h$ automorphism of $H$, $\inn(h)(x)=hxh^{-1}$. We write $X_\bullet(H), X^\bullet(H)$ for the cocharacter and character groups (over a choice of algebraic closure of $K$). 

\section{Abstract arguments: toward compatible systems associated to superrigid representations}\label{abstract} 

\begin{lem}\label{descentlem}
Let $\Pi$ be a profinite group, let $\Pi'$ be a closed normal subgroup, and let $H$ be an adjoint algebraic group over $\Qlb$. Suppose that $\rho \colon \Pi' \to H(\Qlb)$ is a continuous homomorphism with Zariski-dense image in $H$ such that for all $x \in \Pi$, the homomorphism $\rho^x \colon \Pi' \to H(\Qlb)$ given by $\rho^x(h)= \rho(xhx^{-1})$ is $H(\Qlb)$-conjugate to $\rho$: that is, there exists $h_x \in H(\Qlb)$ such that $\rho^x=h_x \rho h_x^{-1}$. Then $h_x$ is unique, and $\rho$ has a unique extension to a homomorphism $\rho \colon \Pi \to H(\Qlb)$; this extension satisfies $\rho(x)=h_x$ for all $x \in \Pi$.
\end{lem}
\begin{proof}
The centralizer of $\rho$ in $H$ is trivial, so $h_x$ is uniquely determined for each $x \in \Pi$. For $x\in \Pi'$, $h_x= \rho(x)$, so the assignment $x \mapsto h_x$ extends $\rho$. It defines a homomorphism on $\Pi$ because $h_x$ is unique and $\rho^{xy}= (\rho^x)^y$.
\end{proof}

Let $F \subset \CC$ be a number field, let $X_F$ be a smooth geometrically connected scheme of finite type over $F$, and let $\bar{x} \in X_F(\ol{\QQ})$ be a geometric point. There exists an open (dense) immersion $X_F \subset \bar{X}_F$ into a smooth projective variety $\bar{X}_F$ over $F$, with $D_F= \bar{X}_F \setminus X_F$ moreover a strict normal crossings divisor. For some $N \in \ZZ$, the quadruple $(X_F, \bar{X}_F, D_F, \bar{x})$ spreads out to data $(X, \bar{X}, D, \bar{x})$ over $\mc{O}_F[1/N]$, where $\bar{X} \to \Spec(\mc{O}_F[1/N])$ is smooth projective with geometrically connected fibers, $D= \bar{X} \setminus X$ is a relative strict normal crossings divisor,  and $\bar{x} \in X(\ol{\ZZ}[1/N])$. We can also view $\bar{x}$ as a point $x$ of the manifold $X^{\mr{an}}:= X(\CC)^{\mr{an}}$.

\begin{defn}\label{superrigdef}
Let $G$ be an {adjoint} reductive group over $\QQ$, and let $\Gamma$ be a discrete group. We say that $\Gamma$ is $G$-superrigid if for any algebraically closed field $\Omega$ of characteristic zero and any two homomorphisms $\rho_1, \rho_2 \colon \Gamma \to G(\Omega)$ with Zariski-dense image, there exists an automorphism $\tau \in \Aut(G_{\Omega})$ such that $\rho_1= \tau(\rho_2)$.
\end{defn}

\textbf{We assume for the remainder of  \cref{abstract} that $G$ is an adjoint group over $\QQ$, and that the fundamental group $\Gamma:= \pi_1^{\mr{top}}(X^{\mr{an}}, x)$ of $X^{\mr{an}}$ is $G$-superrigid.}

Let $\rho \colon \Gamma \to G(\CC)$ be a homomorphism with Zariski-dense image. Then the arguments of \cite{esnault-groechenig:rigid}, \cite{klevdal-patrikis:G-rigid} (a variant of which we use in Proposition \ref{nearcompatible}) show that $\rho$ is integral, i.e.\ is conjugate to a $G(\mc{O}_L)$-valued representation for some number field $L \subset \overline{\QQ}$. For each place $\lambda$ of $\overline{\QQ}$, we obtain a unique continuous extension $\Gamma \to \pi_1(X_{\CC},\bar{x}) \to G(\ol{\ZZ}_\lambda)$ of $\rho$, or equivalently $\rho_{\lambda, \ol{\QQ}} \colon \pi_1(X_{\ol{\QQ}}, \bar{x}) \to G(\ol{\ZZ}_{\lambda})$. The outer action of $\mr{Gal}(\ol{\QQ}/F)$ on $\pi_1(X_{\ol{\QQ}}, \bar{x})$ induces an action on the set of conjugacy classes of continuous, Zariski dense homomorphisms $\pi_1(X_{\ol{\QQ}}, \bar{x}) \to G(\ol{\QQ}_\lambda)$, which is a torsor under the finite group $\mr{Out}(G_{\ol{\QQ}_\lambda})$ by $G$-superrigidity. Thus there is a (necessarily unique by \cref{descentlem}) continuous extension $\rho_{\lambda, F(\lambda)}$ of $\rho_{\lambda, \ol{\QQ}}$ to $\pi_1(X_{F(\lambda)}, \bar{x})$, with $F(\lambda)$ the fixed field of the stabilizer of the conjugacy class of $\rho_{\lambda, \ol{\QQ}}$.

\begin{prop}\label{nearcompatible}
Assume that $F(\lambda) = F$ for all finite places $\lambda$ of $\ol{\QQ}$, and write $\rho_\lambda = \rho_{\lambda, F} \colon \pi_1(X_F,\bar{x}) \to G(\ol{\QQ}_\lambda)$. Then there exists an integer $N'$ such that
\begin{enumerate}
    \item For each finite place $\lambda$ of $\ol{\QQ}$ the representation $\rho_{\lambda}$ extends uniquely to (i.e., factors through) a homomorphism
    \[ 
        \pi_1(X_{\mc{O}_{F}[1/N'\ell]}, \bar{x}) \to G(\ol{\QQ}_{\lambda}) 
    \]
    where $\ell$ is the rational prime lying under $\lambda$. 
    \item For any two finite places $\lambda$, $\lambda'$ of $\ol{\QQ}$ lying over rational primes $\ell, \ell'$, and all primes $v$ of $\mc{O}_{F}[1/(N' \ell \ell')]$, there exists an automorphism $\tau= \tau(\lambda, \lambda', v)$ of $G_{\ol{\QQ}_{\lambda'}}$ such that $\rho_{\lambda, v}$ and $\tau(\rho_{\lambda', v})$ are companions, where $\rho_{\lambda, v}, \rho_{\lambda',v}$ are the restrictions of $\rho_{\lambda}, \rho_{\lambda'}$ to $\pi_1(X_{\kappa(v)}, \bar{x}_v)$. 
\end{enumerate}
\end{prop}
\begin{proof}
Let $R \colon G_{\ol{\QQ}} \to \GL_{n, \ol{\QQ}}$ be a faithful irreducible representation. For each finite place $\lambda$ of $\ol{\QQ}$, the composition $R \circ \rho_{\lambda}$ is a geometrically irreducible rank $n$ local system on $X_F$, and hence by \cite[Proposition 6.1]{petrov:derham} it (and consequently $\rho_\lambda$) factors uniquely through $\pi_1(X_{\mc{O}_F[1/N_\lambda]},\bar{x})$ for some integer $N_\lambda$ depending on $\lambda$, where we may assume the rational prime below $\lambda$ divides each $N_{\lambda}$. 
The content of (1) is that $N_\lambda$ can be chosen independently of $\lambda$; we will show that $N' = \gcd_\lambda(N_\lambda)$ works (indexed over all finite places of $\ol{\QQ}$). 

Fix a finite place $\lambda$ of $\ol{\QQ}$ lying over a rational prime $\ell$. We suppose that $p$ is a prime that does not divide $N' \ell$ but does divide $N_\lambda$. 

Let $v$ be a place of $F$ above $p$.
Since $p \nmid N'$, there exists a place $\lambda'$ of $\ol{\QQ}$ such that $p \nmid N_{\lambda'}$ (so $\lambda' \nmid p$). We may then define $\rho_{\lambda', v}$ as the restriction $\rho_{\lambda'}|_{\pi_1(X_{\kappa(v)}, \bar{x}_v)}$, and $\rho_{\lambda', \bar{v}}$ as its further restriction to $\pi_1(X_{\ol{\kappa(v)}}, \bar{x}_v)$. 

The tame specialization map $\mr{sp}$ is induced by pullback on finite \'etale (tame) covers  
\begin{equation}\label{eqn:tame-spec}
\xymatrix{
\pi_1(X_{\ol{\QQ}}, \bar{x}) \ar@/^2pc/[rr]^{\mr{sp}} \ar@{->>}[r] & \pi_1(X_{\mc{O}^{\mr{sh}}_{(v)}}, \bar{x}) & \pi_1^t(X_{\ol{\kappa(v)}}, \bar{x}_v) \ar[l]_{\sim},
}
\end{equation}
where $\mc{O}_{(v)}^{\mr{sh}}$ is the strict Henselization of the local ring $\mc{O}_{(v)}:= \mc{O}_{F, v}$ (the localization, not the completion).\footnote{Note that $\mr{F\acute{E}t}^t(X_{\mc{O}_{(v)}^{\mr{sh}}})=\mr{F\acute{E}t}(X_{\mc{O}_{(v)}^{\mr{sh}}})$ since $\mc{O}_{(v)}$ has (generic) characteristic zero, and the components of $D$ are flat. Specialization maps are discussed in \cite[Expos\'e X]{sga1}, and the surjective tame specialization is discussed \cite[Corollary A.12]{lieblich-olsson-pi1} for schemes smooth over complete DVRs with separably closed residue fields. In the case that we use, existence of the (surjective) tame specialization map is well-known to experts, and an exposition in this setting can be found in the initial arXiv version \cite[\S 2]{klevdal-patrikis:SVcompatiblearXiv} of this paper.} Since $\rho_{\lambda'}$ factors through $\pi_1(X_{\mc{O}_{(v)}}, \bar{x})$, we deduce that the image of $\rho_{\lambda',\bar{v}}$ (and hence of $\rho_{\lambda',v}$) is Zariski-dense in $G$, from the corresponding property for $\rho_{\lambda', \ol{\QQ}}$. 

By \cite{drinfeld:pross}, $\rho_{\lambda', v}$ has a $\lambda$-companion (in fact, the combination of \cite{drinfeld:deligneconj} and \cite{chin:indl} suffices without the delicate arguments of \cite{drinfeld:pross} that handle non-connected monodromy groups), which we denote by $\rho_{\lambda' \leadsto \lambda, v}$. Since $\rho_{\lambda', v}$ is tame, the companion $\rho_{\lambda' \leadsto \lambda, v}$ is also tame, by combining \cite[Proposition 4.2]{kerz-schmidt:tameness} with  (\cite[Th\'eor\`eme 9.8]{deligne:constantes}) as in \cite[Theorem 1.1]{esnault-groechenig:rigid}.

The restriction $\rho_{\lambda' \leadsto \lambda, \bar{v}} \colon \pi_1^t(X_{\ol{\kappa(v)}}, \bar{x}_v) \to G(\ol{\QQ}_{\lambda})$ can then be pulled back along the tame specialization map to a homomorphism
\[
\rho_{\lambda' \leadsto \lambda, \bar{v}}^{\mr{top}} \colon \Gamma \to \pi_1(X_{\CC}, \bar{x}) \xrightarrow{\sim} \pi_1(X_{\ol{\QQ}}, \bar{x}) \xrightarrow{\mr{sp}} \pi_1^t(X_{\ol{\kappa(v)}}, \bar{x}_v) \xrightarrow{\rho_{\lambda' \leadsto \lambda, \bar{v}}} G(\ol{\QQ}_{\lambda}),
\]
which also has Zariski-dense image (this reduces to the fact that $\rho_{\lambda' \leadsto \lambda, \bar{v}}$ has Zariski-dense image, which follows for $\rho_{\lambda' \leadsto \lambda, v}$ by \cite{drinfeld:pross} and for $\rho_{\lambda' \leadsto \lambda, \bar{v}}$ because the Zariski-closure of its image is a normal subgroup $H \subset G$, hence a product of simple factors of $G$ since $G$ is adjoint, with the subgroup generated by a single Frobenius element dense in $G/H$; this forces $H=G$). Since $\Gamma$ is $G$-superrigid, $\rho_{\lambda' \leadsto \lambda, \bar{v}}^{\mr{top}}= \tau(\rho)$ for some $\tau \in \Aut(G_{\ol{\QQ}_{\lambda}})$ (here we abusively also denote by $\rho$ the composite $\Gamma \xrightarrow{\rho} G(\mc{O}_L) \subset G(\ol{\QQ}_{\lambda})$). We deduce (for any $v \vert p$) that $\rho_{\lambda, \ol{\QQ}}$ also factors through $\pi_1(X_{\ol{\QQ}}, \bar{x}) \twoheadrightarrow \pi_1^t(X_{\ol{\kappa(v)}}, \bar{x}_v)$, since $\rho= \tau^{-1}(\rho^{\mr{top}}_{\lambda' \leadsto \lambda, \bar{v}})$ does. 

Let $X'=X_{\mc{O}_F[1/N_{\lambda}^{(p)}]}$, where $N_{\lambda}^{(p)}$ denotes the prime-to-$p$ part of $N_{\lambda}$. 
We now show $\rho_{\lambda}$ extends to $X'$ (using a variant of the way \cite[Proposition 6.1]{petrov:derham} exploits geometric irreducibility). Consider any morphism $\varphi \colon \Spec(\mc{O}_L[1/S_{\varphi}]) \to X'$, for $L$ a number field and $S_{\varphi}$ a finite set of places of $L$ disjoint from those above $p$, such that $\varphi^{-1}(X_{\mc{O}_F[1/N_{\lambda}]}) \neq \emptyset$. For $w \vert p$ a place of $L$ and a choice of place $\ol{w} \vert w$ of $\ol{\QQ}$ (hence of $\mc{O}_{L, (w)}^{\mr{sh}} \subset \ol{\QQ}$), $\varphi$ induces generic and special fiber base points and a specialization homomorphism $\mr{sp}_{\varphi} \colon \pi_1(X_{\ol{\QQ}}, \bar{\varphi}) \twoheadrightarrow \pi_1^t(X_{\ol{\kappa(w)}}, \bar{\varphi}_w)$. We write $\rho_{\lambda, \varphi}$ for the representation of $\pi_1(X_{\mc{O}_F[1/N_{\lambda}]}, \bar{\varphi})$ induced (uniquely up to inner automorphism) by $\rho_{\lambda}$ and $\varphi$. Then we claim $\rho_{\lambda, \varphi}(I_{\bar{w}})=1$, where $I_{\bar{w}}$ is the image using the section $\varphi$. The key point is that $\mr{sp}_{\varphi}$ is $I_{\bar{w}}$-equivariant (as noted in the proof of \cite[Theorem 1.1.3]{litt:arithmetic2} for the analogous case of prime-to-$p$ specialization). The action on the target is trivial, and since $\rho_{\lambda, \varphi}|_{X_{\ol{\QQ}}}$ factors through $\mr{sp}_{\varphi}$ (just as previously for $\bar{x}$), we conclude that $\rho_{\lambda, \varphi}(I_{\bar{w}})$ commutes with $\rho_{\lambda, \varphi}(\pi_1(X_{\ol{\QQ}}, \bar{\varphi}))$. The latter is Zariski-dense in $G$, so $\rho_{\lambda, \varphi}(I_{\bar{w}})=1$. More intrinsically, this tells us that the lisse sheaf $\mc{F}$ associated to $\rho_{\lambda, \varphi}$ (in a fixed faithful representation of $G$), hence also associated to $\rho_{\lambda}$, has pullback $\varphi^*(\mc{F})$ that is unramified at $w \in \Spec(\mc{O}_L[1/S_{\varphi}])$ (this conclusion is independent of the variation of base points).

To finish the proof, we use \cite[Corollary 5.2]{drinfeld:deligneconj}. Suppose that $\rho_{\lambda}$ does not extend to $X'$. By \textit{loc. cit.}, there is a closed point $z \in X' \setminus X= \sqcup_{v \vert p} X_{\kappa(v)}$ and a rank 1 subspace $V \subset T_z(X')$ satisfying: 

$(\star)$ For any $\varphi$ as above such that $\varphi(w)=z$ and the image of the tangent map 
\[
T_w(\Spec(\mc{O}_{L}[1/S_{\varphi}]) \to T_z(X')
\]
is $V \otimes_{\kappa(z)} \kappa(w)$, $\varphi^*(\mc{F})$ (a lisse sheaf on $\varphi^{-1}(X_{\mc{O}_F[1/N_{\lambda}]})$) is ramified at $w$.

As then in the proof of \cite[Lemma 5.3]{drinfeld:deligneconj}, an appropriate form of Hilbert irreducibility (Theorem 2.15 of \textit{loc. cit.} shows much more) shows a $\varphi$ satisfying the hypotheses of $(\star)$ exists. We have however shown that in our context any such $\varphi$ is unramified at $w$, a contradiction. We conclude that $\rho_{\lambda}$ extends to $X'$.

(2) The argument just given tells us more: $\tau(\rho^{\mr{top}}_{\lambda, \bar{v}})= \rho^{\mr{top}}_{\lambda' \leadsto \lambda, \bar{v}}$, and thus $\tau(\rho_{\lambda, \bar{v}}) = \rho_{\lambda' \leadsto \lambda, \bar{v}}$, so Lemma \ref{descentlem} implies that $\tau(\rho_{\lambda, v}) = \rho_{\lambda' \leadsto \lambda, v}$, since these are extensions, necessarily unique, of the corresponding $\pi_1(X_{\ol{\kappa(v)}}, \bar{x}_v)$ representations. 
\end{proof}

\begin{rmk} In our Shimura variety application we will have a superrigid $\rho$ for which the hypothesis that $F(\lambda)$ is independent of $\lambda$ follows from the theory of canonical models. We expect that $F(\lambda)$ may be taken independently of $\lambda$ whenever $\Gamma$ is $G$-superrigid. In fact, the arithmetic descent preceding Proposition \ref{nearcompatible} works when $\rho$ is any $G$-cohomologically rigid local system with quasi-unipotent monodromy local monodromy and (for simplicity) Zariski-dense image (handling the local monodromy as in \cite{esnault-groechenig:rigid}), and if the fields $F(\lambda)$ obtained for each of the finitely many local systems satisfying the same hypotheses as $\rho$ are independent of $\lambda$, then the proof of part (1) of \ref{nearcompatible} works.
\end{rmk}

\begin{rmk}\label{easycompanion}
Since $G$ is connected reductive and all $\rho_{\lambda, v}$ have Zariski-dense image, the companion property is in fact equivalent to: for a density one set of closed points $y \in |X_{\kappa(v)}|$, the semisimple parts of $\rho_{\lambda, v}(F_y)$ and $\tau(\rho_{\lambda', v})(F_y)$ are conjugate to elements of $G(\ol{\QQ})$, and then define the same $G(\ol{\QQ})$-conjugacy class. (For this equivalence, which is in general false without the restriction on the monodromy groups, see \cite[Proposition 6.4]{bhkt:fnfieldpotaut}---strictly speaking, we must generalize \textit{loc. cit.}, which treats the curve case, to the higher-dimensional case using the \v{C}ebotarev density theorem for $X_{\kappa(v)}$.)
\end{rmk}

\begin{rmk}
Unfortunately, this argument does not seem to yield even a single $v$ such that $\rho_{\lambda, v}$ and $\rho_{\lambda', v}$ are companions. It does however show that (in the setup of part (2)), fixing $\lambda$ and $\lambda'$, we can partition the places of $\mc{O}_{F}[1/N'\ell \ell']$ into sets $\{\Delta_\tau\}_{\tau \in \Out(G^{\mr{ad}})}$ such that for $v \in \Delta_\tau$, $\rho_{\lambda, v}$ and $\tau(\rho_{\lambda', v})$ are companions. Since $\Out(G)$ is finite, some set $\Delta_{\tau}$ has positive upper-density. If we knew that $\rho_{\lambda}$ had \textit{some} $\lambda'$-companion $\rho_{\ltolp}$, we could in fact deduce that $\rho_{\ltolp} \cong \tau(\rho_{\lambda'})$ ($\cong$ here denotes $G$-conjugacy). Indeed, this follows from a suitable generalization of the results of \cite{rajan:sm1} (making crucial use in the application of Zariski-density of the images). Of course, we expect this to hold with $\tau=\mr{id}$!
\end{rmk}

In the application to Shimura varieties, we will control the automorphism $\tau$ using our additional knowledge that $\rho_{\lambda}$ and $\rho_{\lambda'}$ are compatible when restricted to certain horizontal sections of $X$ provided by the special points. In the remainder of this section, we abstract the relevant argument.

\begin{assumption}\label{horizontal}
Continue with the setting of Proposition \ref{nearcompatible}. Let $\lambda, \lambda'$ be finite places of $\ol{\QQ}$ lying over rational primes $\ell, \ell'$. Fix a prime $v$ of $\mc{O}_{F}[1/(N' \ell \ell')]$, and suppose there is a point $y \in X(\mc{O}_{C_y}[1/N_y])$, where $C_y/F$ is a finite extension, and $N_y$ is some multiple of $N'$ that is not divisible by $v$, and the following holds:
\begin{enumerate}
    \item The specializations $\rho_{\lambda, y} \colon \pi_1(C_y) \to G(\ol{\QQ}_{\lambda})$ and $\rho_{\lambda', y} \colon \pi_1(C_y) \to G(\ol{\QQ}_{\lambda'})$ are compatible at any place $v_y$ of $C_y$ above $v$. Since by construction they are unramified at such $v_y$, this condition means that $\rho_{\lambda, y}(\Frob_{v_y})$ and $\rho_{\lambda', y}(\Frob_{v_y})$ have semisimple parts $t_{\lambda, v_y}$ and $t_{\lambda', v_y}$ whose conjugacy classes are defined over $\ol{\QQ}$ and in $G(\ol{\QQ})$ define the same conjugacy class. We denote by $[t_{v_y}] \in [G \git G](\ol{\QQ})$ this common underlying semisimple conjugacy class, a closed point of the GIT quotient $[G \git G]$.
    \item The element $t_{\lambda, v_y} \in G(\ol{\QQ}_{\lambda})$ is not fixed by any automorphism of $G_{\ol{\QQ}_{\lambda}}$ with non-trivial image in the outer automorphism group $\Out(G_{\ol{\QQ}_{\lambda}})$. \qed
\end{enumerate}

\end{assumption}

\begin{prop}\label{symbreak}
For each prime $v$ of $\mc{O}_{F}[1/(N'\ell \ell')]$ such that Assumption \ref{horizontal} holds for $v$, the representations $\rho_{\lambda, v}$ and $\rho_{\lambda', v}$ of $\pi_1(X_{\kappa(v)}, \bar{x}_v)$ are companions.
\end{prop}
\begin{proof}
By Proposition \ref{nearcompatible} there exists an automorphism $\tau$ of $G_{\ol{\QQ}_{\lambda'}}$ such that $\rho_{\lambda, v}$ and $\tau(\rho_{\lambda',v})$ are companions. The action of $\Aut(G)$ on $G$ factors uniquely through an action of $\Out(G)$ on $[G \git G]$ such that for any algebraically closed field $\Omega$, $\sigma \in \Aut(G_\Omega)$, and $g \in G(\Omega)$ we have $[\sigma g] = \bar{\sigma}[g]$ for $\bar{\sigma}$ the image of $\sigma$ in $\Out(G_\Omega)$. Combining this with part (1) of Assumption \ref{horizontal}, we see 
\begin{equation}\label{companioneqtn}
    [t_{v_y}] = [\rho_{\lambda, v}(\Frob_{v_y})] = \bar{\tau}[\rho_{\lambda',v}(\Frob_{v_y})] = \bar{\tau}[t_{v_y}],
\end{equation}
for $\bar{\tau}$ the image of $\tau$ in $\mr{Out}(G_{\ol{\QQ}_{\lambda'}}) \cong \mr{\Out}(G_{\ol{\QQ}})$. 

Let $\tilde{\tau} \in \Aut(G_{\ol{\QQ}})$ induce $\bar{\tau}$, and let $t_{v_y} \in G(\ol{\QQ})$ be a semisimple element in the fiber of $G \to [G \git G]$ over $[t_{v_y}]$; $t_{v_y}$ is well-defined up to $G(\ol{\QQ})$-conjugation. Since $\tilde{\tau}(t_{v_y})$ and $t_{v_y}$ are both semisimple and lie over $[t_{v_y}]$ (Equation \eqref{companioneqtn}), there exists $g \in G(\ol{\QQ})$ such that $(\inn(g)\circ\tilde{\tau})(t_{v_y})=t_{v_y}$. Now by construction $t_{\lambda, v_y}= \inn(h_{\lambda})t_{v_y}$ for some $h_{\lambda} \in G(\ol{\QQ}_{\lambda})$, so $t_{\lambda, v_y}$ is then fixed by the automorphism $\inn(h_{\lambda}g)\tilde{\tau} \inn(h_{\lambda}^{-1})$ of $G_{\ol{\QQ}_{\lambda}}$. Thus by part (2) of Assumption \ref{horizontal}, the image of $\tilde{\tau}$ in $\Out(G_{\ol{\QQ}})(\ol{\QQ}) \xrightarrow{\sim} \Out(G_{\ol{\QQ}_{\lambda}})(\ol{\QQ}_{\lambda})$ is trivial.

\end{proof}

\section{The case of Shimura varieties}\label{svsection}
\subsection{Canonical models and local systems}\label{localsysdef}
Let $(G, X)$ be a Shimura datum, For any neat compact open subgroup $K \subset G(\A_f)$, we write $\Sh_K(G,X)$ for the complex Shimura variety with $\Sh_K(G, X)(\CC)= G(\QQ) \backslash (X \times G(\A_f))/K$ and $\Sh(G,X) = \varprojlim_K \Sh_K(G,X)$ for the limit over all compact open subgroups. We assume for simplicity that $Z_G(\QQ)$ is discrete in $Z_G(\A_f)$; this does not entail any significant loss of generality for the study of Galois representations, as it is automatic if $G$ is the generic Mumford-Tate group on $X$: see \cite[Appendix]{ullmo-yafaev:generalized}. It then follows that $\Sh(G, X)(\CC)$ identifies with $G(\QQ) \backslash X \times G(\A_f)$; we denote the element represented by an ordered pair $(x, a) \in X \times G(\A_f)$ by $[x, a]$, and we let $[x, a]_K$ denote its image in $\Sh_K(G, X)(\CC)= G(\QQ) \backslash (X \times G(\A_f))/K$. 

Let $E=E(G, X)$ be the reflex field of $(G, X)$. The theory of canonical models (\cite{deligne:canonicalmodels}, \cite{milne:canonical}) produces an inverse system $(\Sh_K)_K$, indexed over (neat) compact open subgroups of $G(\A_f)$, and equipped with right $G(\A_f)$-action, of smooth quasi-projective varieties over $E$ such that: (1) $(\Sh_K(\CC))_K$ with its $G(\A_f)$-action is isomorphic to $(\Sh_K(G, X)(\CC))_K$ with its $G(\A_f)$-action in such a way that (2) for every special pair $(T_x, x) \subset (G, X)$ and $a \in G(\A_f)$, $[x, a]_K \in \Sh_K(E(x)^{\mr{ab}})$ and
\begin{equation}\label{reciprocity}
\sigma[x, a]_K= [x, r_x(\alpha)a]_K
\end{equation}
whenever the reciprocity map $\rec_{E(x)} \colon \A_{E(x)}^\times \to \Gamma_{E(x)}^{\mr{ab}}$ satisfies $\rec_{E(x)}(\alpha)= \sigma$. Let us explain the notation: 
\begin{itemize}
\item By definition $x$ is a homomorphism $h_x \colon \mathbb{S} \to G_{\RR}$ factoring through $T_{x, \RR} \subset G_{\RR}$ for some $\QQ$-torus $T_x \subset G$. 
\item Restriction along $\mathbb{G}_m \xrightarrow{z \mapsto (z, 1)} \mathbb{S}_{\CC}$ of $h_{x, \CC}$ yields the associated cocharacter $\mu_x \in X_{\bullet}(T_x)$ over $\CC$, whose field of definition is the number field $E(x) \subset \CC$. 
\item The homomorphism $r_x$ is as in \cite[Definition 12.8]{milne:shimuraintro}: we take the canonical map $\Res_{E(x)/\QQ} \mathbb{G}_m \xrightarrow{\Res_{E(x)/\QQ}(\mu_x)} \Res_{E(x)/\QQ} (T_x \otimes E(x)) \xrightarrow{N_{E(x)/\QQ}} T_x$ evaluated on $\A$-points and then project to $T_x(\A_f)$ to define $r_x \colon \A_{E(x)}^\times \to T_x(\A_f)$. 
\end{itemize}
Note that for any finite extension $E'/E(x)$, we can regard $\mu_x$ as defined over $E'$, and then in the same way we can define a homomorphism $r_{x, E'} \colon \mathbb{A}^\times_{E'} \to T_x(\A_f)$; \cite[Equation 2.2.2.1]{deligne:canonicalmodels} records the compatibility $r_{x, E'}= r_{x} \circ N_{E'/E(x)}$. For $\sigma \in \Gamma_{E'}$ we then see (using norm-compatibility of the global reciprocity maps) that Equation \eqref{reciprocity} holds taking $r_{x, E'}$ in place of $r_x$ and taking $\alpha' \in \A_{E'}^\times$ satisfying $\rec_{E'}(\alpha')= \sigma$ in place of $\alpha$. 
\begin{convention}\label{reclazy}
We will eventually use Equation \eqref{reciprocity} in the more flexible form just described, and we will allow ourselves the notational abuse of writing $r_x(\alpha')$ for $r_{x, E'}(\alpha')$ when $\alpha' \in \mathbb{A}_{E'}^{\times}$ (see e.g. the statement of Lemma \ref{specialgalois}).
\end{convention}

Now we fix a neat compact open subgroup $K_0 \subset G(\A_f)$. Following \cite[\S 4]{cadoret-kret:galois-generic}, we recall the construction of the canonical $K_0$-valued local systems on geometrically-connected components of $\Sh_{K_0}$. Fix a point $s \in \Sh(\CC)$, with image $sK \in \Sh_K(\CC)$ for each $K$, and let $S_{K, s}$ be the geometrically connected component of $\Sh_{K}$ containing $sK$; it is defined over some finite abelian extension $E_{K, s}$ of $E$. In particular, for any open normal subgroup $K \subset K_0$, we have the following Cartesian squares of (non-connected) Galois covers with group $K_0/K$ (acting on the right):
\[
\xymatrix{
\Sh_K \ar[d]^{K_0/K} & \Sh_K \otimes_E E_{K_0, s} \ar[l] \ar[d]^{K_0/K} \ar[l] & \Sh_K \times_{\Sh_{K_0}} S_{K_0, s} \ar@{_{(}->}[l] \ar[d]^{K_0/K} & \widetilde{S}_{K, s} \ar[ld] \ar@{_{(}->}[l] \\
\Sh_{K_0} & \Sh_{K_0} \otimes_E E_{K_0, s} \ar[l] & S_{K_0, s} \ar@{_{(}->}[l]& \\
}
\]
where $\widetilde{S}_{K, s}$ is by definition the connected component of $\Sh_K \otimes_E E_{K_0, s}$ containing $sK$. The fact that the vertical maps are Galois covers with group $K_0/K$ uses the assumptions that $K_0$ is neat and that $Z_G(\QQ)$ is discrete in $Z_G(\A_f)$: see \cite[Lemma 2.1]{ullmo-yafaev:generalized}. 
To simplify notation, in our fundamental group notation we will simply write $\bar{s}$ for any of these geometric points $sK$. Thus from the connected pointed Galois cover $(\widetilde{S}_{K, s}, \bar{s}) \to (S_{K_0, s}, \bar{s})$ we obtain a canonical surjection
\[
\pi_1(S_{K_0, s}, \bar{s}) \twoheadrightarrow \Aut(\widetilde{S}_{K, s}/S_{K_0, s})^{\mr{op}} \subset K_0/K.
\] 
Let us recall precisely how this works, since we will want to compute the restriction to special points of these homomorphisms. Let $F_{\bar{s}} \colon \mr{F\acute{E}t}_{S_{K_0, s}} \to \mr{Sets}$ be the fiber functor at $\bar{s}$ on the category of finite \'{e}tale covers of $S_{K_0, s}$, so $\pi_1(S_{K_0, s}, \bar{s})$ is by definition the profinite group $\Aut(F_{\bar{s}})$. Once we choose representatives, indexed by a poset $I$, of each isomorphism class of connected Galois covers $(S_i \to S_{K_0, s})_{i \in I}$ and fix elements $\bar{s}_i \in F_{\bar{s}}(S_i)$, for all $i \geq i'$ there is a unique morphism $f_{i i'} \colon S_i \to S_{i'}$ over $S_{K_0, s}$ carrying $\bar{s}_i$ to $\bar{s}_{i'}$. This data induces an isomorphism of functors
\[
F_{\bar{s}}(\bullet) \xrightarrow{\sim} \mr{colim}_I \Hom_{\mr{F\acute{E}t}_{S_{K_0, s}}}(S_i, \bullet),
\]
and $A= \mr{lim}_I \Aut(S_i/S_{K_0, s})$ (automorphisms acting on the left of each $S_i$) canonically acts on the right of this colimit, hence we obtain a homomorphism $A^{\mr{op}} \to \Aut(F_{\bar{s}})$, which is an isomorphism. In particular in our setting we take the covers $\widetilde{S}_{K, s}$ among the $S_i$ and always for these choose the base-points $\gamma_i= \bar{s}K$, and we obtain the surjection $\pi_1(S_{K_0, s}, \bar{s}) \to \Aut(\widetilde{S}_{K, s}/S_{K_0, s})^{\mr{op}}$, which via the above diagram is canonically included in $K_0/K$.

Taking the limit over all such $K \subset K_0$, we obtain the representation
\[
\rho_{K_0, s} \colon \pi_1(S_{K_0, s}, \bar{s}) \to K_0
\]
and its projection $\rho_{K_0, s, \ell}$ to the image $K_{0, \ell}$ of $K_0$ in $G(\QQ_{\ell})$.

We now describe how the fundamental property \eqref{reciprocity} of canonical models determines the specialization of these local systems to special points. Let $s=[x, a] \in \Sh(\CC)$ be a special point in the sense that the Mumford-Tate group of the Hodge structure $h_x \colon \mathbb{S} \to G_{\RR}$ corresponding to $x \in X$ is a torus $T_x$. The image $sK_0 \in S_{K_0, s}$ is defined over an abelian extension $E(sK_0)$ of $E(x)$ inside $\overline{\QQ}$, and we write $\bar{s}$ for either of the associated $\ol{\QQ}$-valued or $\CC$-valued geometric points. The next lemma follows from the definition of $\rho_{K_0, s}$ given above:
\begin{lem}\label{specialgalois}
Let $s=[x, a] \in \Sh(\CC)$ be a special point as above. Then the restriction of $\rho_{K_0, s}$ along the specialization $\pi_1(\Spec(E(sK_0)), \bar{s}) \to \pi_1(S_{K_0, s}, \bar{s})$ maps $\sigma \in \Gamma_{E(sK_0)}^{\mr{ab}}$ to the unique $k_0(\sigma) \in K_0$ such that there exists $q \in G(\QQ)_+$ satisfying $qx=x$ and $qak_0(\sigma)=r_x(\alpha)a$, where again $\alpha \in \mathbb{A}_{E(sK_0)}^\times$ is any idele satisfying $\rec_{E(sK_0)}(\alpha)= \sigma$ (and we observe Convention \ref{reclazy}).
\end{lem}
\begin{proof}
Let $K$ as above be a finite-index normal subgroup of $K_0$. Then for any for any $\sigma \in \Gamma_{E(sK_0)}^{\mr{ab}}$, $\sigma[x, a]_K= [x, a]_K k_0(\sigma)$ for a unique $k_0(\sigma) \in K_0/K$. By construction the class $k_0(\sigma) K$ is the image of $\rho_{K_0, s}(\sigma)$ modulo $K$. The lemma follows now from (see Equation \eqref{reciprocity} and the discussion following) the equality $\sigma[x,a]_K = [x, r_x(\alpha)a]_K$, where we have allowed the notational abuse of Convention \ref{reclazy}, writing $r_x$ instead of $r_x \circ N_{E(sK_0)}/E(x)= r_{x, E(sK_0)}$. 
\end{proof}

\subsection{Integral models}\label{integralmodels}
Throughout this section we additionally assume that $G$ is a connected reductive group over $\QQ$ whose adjoint group $G^{\mr{ad}}$ is $\QQ$-simple and has real rank $\mr{rk}_{\RR}(G^{\mr{ad}}_{\RR})$ at least 2. Since $(G, X)$ is a Shimura datum, there exists a totally real field $F$ and an absolutely simple adjoint group $G^s$ over $F$ such that $G^{\mr{ad}} \cong \Res_{F/\QQ}(G^s)$ (\cite[2.3.4(a)]{deligne:canonicalmodels}).  
We once and for all fix a connected component $X^+$ of $X$, and for any compact open subgroup $K \subset G(\A_f)$ we let $\mc{C}_K$ be a set of representatives in $G(\A_f)$ of the double coset space $G(\QQ)_{+} \backslash G(\A_f)/K$. Then the complex manifold $\Sh_K(G, X)(\CC)= G(\QQ) \backslash (X \times G(\A_f))/K$ is a disjoint union of locally symmetric spaces (\cite[Lemma 5.13]{milne:shimuraintro}):
\begin{equation}\label{components}
G(\QQ) \backslash (X \times G(\A_f))/K= G(\QQ)_+ \backslash (X^+ \times G(\A_f))/K= \bigsqcup_{g \in \mc{C}_K} \Gamma_g \backslash X^+, 
\end{equation}
where the map sends $x \in \Gamma_g \backslash X^+$ to $[x, g]_{K}$.

We will use Margulis's superrigidity theorems to check that the criterion of Definition \ref{superrigdef} holds for the fundamental groups of these locally symmetric spaces, namely the images of each $\Gamma_g$ in $G^{\mr{ad}}(\QQ)$.

\begin{prop}\label{svsuperrig}
Let $K$  
be a neat compact open subgroup of $G(\A_f)$, and let $\Gamma$ be the fundamental group of one of the connected components of $\Sh_K(G, X)(\CC)$. Then $\Gamma$ is $G^{\mr{ad}}$-superrigid in the sense of Definition \ref{superrigdef}.
\end{prop}
\begin{proof}
Note that $\Gamma$, the adjoint image of one of the congruence subgroups $\Gamma_g$ above, is an arithmetic subgroup of $G^{\mr{ad}}(\QQ)$ (\cite[]{platonov-rapinchuk}). Let $\Omega$ be any algebraically closed field of characteristic zero, and let $\rho_1, \rho_2 \colon \Gamma \to G^{\mr{ad}}(\Omega)$ be two homomorphisms with Zariski-dense image. Since $G^{\mr{ad}}(\Omega) \cong \prod_{\sigma \colon F \to \Omega} G^s_{\sigma}(\Omega)$, where $G^s_{\sigma}$ denotes the base-change to $\Omega$ along $\sigma$, we obtain for each $\sigma \colon F \to \Omega$ a pair of homomorphisms $\rho_{i, \sigma} \colon \Gamma \to G^s_{\sigma}(\Omega)$ with Zariski-dense image. We will now apply \cite[Theorem VIII.3.4(a)]{margulis:discrete}: since there is a lot of notation in the formulation of \textit{loc. cit.}, we will orient the reader with some explanation of how our setup fits into Margulis's:
\begin{itemize}
\item The field $K$ is in our case $\QQ$, the connected (almost) $K$-simple non-commutative $K$-group $\mathbf{G}$ is our $G^{\mr{ad}}$, the finite set $S$ of places of $K$ is $\{\infty\}$; since $G^{\mr{ad}}_{\RR}$ is not anisotropic, the ($S$-)arithmetic subgroup $\Gamma \subset G^{\mr{ad}}(\QQ)$ belongs to the set $\Omega_{\emptyset}$ defined at the start of \cite[\S VIII.3]{margulis:discrete}. 
\item The absolutely simple adjoint group $\mathbf{H}$ is our $G^s_{\sigma}$ (we apply the theorem separately for each choice of $\sigma$), and the field $l$ is our $\Omega$. 
\end{itemize}
Then \cite[Theorem VIII.3.4(a)]{margulis:discrete} implies that there is a field extension $\Omega' \supset \Omega$ ($\Omega'= l'$ in the notation of \textit{loc. cit.}), and for $i=1, 2$ and each $\sigma$, there is a surjection $\eta_{i, \sigma} \colon G^{\mr{ad}}_{\Omega'} \to G^s_{\sigma, \Omega'}$ of algebraic groups over $\Omega'$\footnote{In \textit{loc. cit.} the base-change to $l'=\Omega'$ requires the proof also to produce an embedding $K \to l'$, there denoted $\sigma$; of course there is no choice when $K= \QQ$. 
Also note that strictly speaking for each $\rho_{i, \sigma}$ we produce an extension $\Omega'$ of $\Omega$, but we can just take our $\Omega'$ to be the compositum in some overfield of the collection of such extensions.} such that $\rho_{i, \sigma}= \eta_{i, \sigma}|_{\Gamma}$; here we restrict $\eta_{i, \sigma}$ along $\Gamma \subset G^{\mr{ad}}(\QQ) \subset G^{\mr{ad}}(\Omega')$. Thus each $\rho_i$ is given by a product
\[
\eta_i \colon G^{\mr{ad}}_{\Omega'} \xrightarrow{\prod_{\sigma \colon F \to \Omega} \eta_{i, \sigma}} \prod_{\sigma \colon F \to \Omega} G^s_{\sigma, \Omega'} \cong G^{\mr{ad}}_{\Omega'}.
\]
Since each $\rho_i$ has Zariski-dense image in $G^{\mr{ad}}_{\Omega'}$, each $\eta_i$ must be an isomorphism. Thus $\eta_1=\tau(\eta_2)$ for some automorphism $\tau$ of $G^{\mr{ad}}_{\Omega'}$, and in particular $\rho_1= \tau(\rho_2)$. It remains to descend $\tau$ to an automorphism of $G^{\mr{ad}}_{\Omega}$. For this descent (not for the conclusion of the proposition!) we may assume $\tau$ is inner, because representatives of $\Out(G^{\mr{ad}}_{\Omega'})$ already exist in $\Aut(G^{\mr{ad}}_{\Omega})$. 
Thus we may assume $\rho_1= g \rho_2 g^{-1}$ for some $g \in G^{\mr{ad}}(\Omega')$. For any $\zeta \in \Aut(\Omega'/\Omega)$, $\zeta (\rho_i)= \rho_i$ for each $i$, so for all $x \in \Gamma$, 
\[
\zeta(g) \rho_2(x) \zeta(g)^{-1}= g \rho_2(x) g^{-1},
\]
implying that $g^{-1}\zeta(g)$ centralizes $\rho_2$, hence (by Zariski-density and that $G^{\mr{ad}}$ is adjoint) $g= \zeta(g)$. Finally we deduce that $g$ lies in $G^{\mr{ad}}(\Omega)$ after all, and so we have verified the $G^{\mr{ad}}$-superrigidity of $\Gamma$. 
\end{proof}
\begin{rmk}
The extension $\Omega'/\Omega$ produced by \cite[Theorem VIII.3.4(a)]{margulis:discrete} is not finite unless $\Omega= \CC$, so the last step of this argument is important.
\end{rmk}

For the rest of this section, we fix a neat compact open subgroup $K_0 \subset G(\A_f)$. We next fix an appropriate integral model of $\Sh_{K_0}(G, X)$. Because of the generality in which we work, we have only the `soft' construction of integral models coming from spreading-out arguments.  
We choose our integral model to parallel the setup of \S \ref{abstract} and will benefit from the special properties of such a model examined in \cite[Theorem 4.1]{pila-shankar-tsimerman:andre-oort}. More precisely, let $E=E(G, X)$ be the reflex field of $(G, X)$. Then \textit{loc. cit.} shows that there is an integer $N \geq 1$ such that $\Sh_{K_0}(G, X)$ has an integral model $\mc{S}_{K_0}$ over $\mc{O}_E[1/N]$ such that every special point of $\Sh_{K_0}(G, X)(\ol{\QQ})$ extends to an element of $\mc{S}_{K_0}(\ol{\ZZ}[1/N])$; this is a weak analogue of the assertion that CM abelian varieties have potentially good reduction everywhere (which suggests it may be possible to take $N=1$). The model $\mc{S}_{K_0}$, as in our abstract setup in \S \ref{abstract} (see the discussion before Definition \ref{superrigdef}), is the complement of a relative strict normal crossings divisor in a smooth projective scheme $\ol{\mc{S}}_{K_0}$ over $\mc{O}_E[1/N]$.

For each element of $\pi_0(\Sh_{K_0}(\CC))$, we fix a special point $s \in \Sh(\CC)$ whose image $sK_0 \in \Sh_{K_0}(\CC)$ belongs to the given component, which we as in \S \ref{localsysdef} denote by $S_{K_0, s}$, a smooth geometrically-connected variety defined over some finite (abelian) extension $E_{K_0, s}$ of the reflex field $E$. From now on we only consider this connected Shimura variety. Our integral model $\mc{S}_{K_0}$ then determines an integral model $\mc{S}_{K_0, s}$ of $S_{K_0, s}$: the closure $\mc{S}_{K_0,s} := \ol{S_{K_0, s}}$ in $\mc{S}_{K_0}$ is generically smooth over $\mc{O}_{E(K_0,s)}$ so after increasing $N$ we can assume $\mc{S}_{K_0, s}$ is smooth over $\mc{O}_{E_{K_0, s}}[1/N]$, and $s$ extends to an element of $\mc{S}_{K_0, s}(\ol{\ZZ}[1/N])$. We for each $\ell$ have the canonical representation
\[
\rho_{K_0, s, \ell} \colon \pi_1(S_{K_0, s}, \bar{s}) \to K_{0, \ell} \subset G(\QQ_{\ell})
\]
described in \S \ref{localsysdef} and whose specialization to the special point $sK_0$ is described in Lemma \ref{specialgalois}. We wish, in combination with Proposition \ref{svsuperrig}, to apply the results of \S \ref{abstract} to these $\ell$-adic representations but must first explain a compatibility between the topological and \'etale fundamental group representations.

We write the basepoint $s=[x, a]$ for some $x \in X^+$ and $a \in G(\A_f)$, which we may take among the representatives $\mc{C}_{K_0}$ of $G(\QQ)_+ \backslash G(\A_f)/K_0$. The complex manifold $S_{K_0, s}(\CC)$ then canonically identifies to $\Gamma_a \backslash X^+$ (see Equation \eqref{components}), and $\pi_1^{\mr{top}}(S_{K_0, s}(\CC), \bar{s})$ (using $x \in X^+$) identifies with the adjoint image $\Gamma:= \Gamma_a^{\mr{ad}} \subset G^{\mr{ad}}(\QQ)$: the former is $\Aut((X^+, x)/(\Gamma_a \backslash X^+, \bar{s}))^{\mr{op}}$, which is $\Gamma$ via the (right!) action $y \mapsto \gamma^{-1}(y)$ ($\gamma \in \Gamma$, $y \in X^+$).\footnote{In fact, $\Gamma_a \to \Gamma$ is an isomorphism since $Z_G(\QQ)$ is discrete in $Z_G(\A_f)$, and $K$ is compact and neat.} Using the identification in Equation \eqref{components}, we can compute the composite homomorphism
\[
\Gamma \xrightarrow{\sim} \pi_1^{\mr{top}}(S_{K_0, s}(\CC), \bar{s}) \to \pi_1(S_{K_0, s}, \bar{s}) \to K_0
\]
by checking its image mod any normal open subgroup $K \subset K_0$: since $[\gamma^{-1}x, a]_K= [x, \gamma a]_K= [x, a]_K \cdot (a^{-1}\gamma a)$, $\gamma \in \Gamma$ maps under the composite to $a^{-1} \gamma a \in K_0$. In particular, the diagram
\[
\xymatrix{
\pi_1^{\mr{top}}(S_{K_0, s}, \bar{s}) \ar[r] \ar[d] & G^{\mr{ad}}(\QQ) \ar[d]^{a^{-1}(\bullet)a} \\
\pi_1(S_{K_0, s}, \bar{s}) \ar[r]_-{\rho^{\mr{ad}}_{K_0, s, \ell}} & G^{\mr{ad}}(\ol{\QQ}_{\ell})
}
\]
commutes, where the top arrow is induced by the inclusion $\Gamma \subset G^{\mr{ad}}(\QQ)$. 

Now, since $G$ is of non-compact type, the Borel density theorem (\cite[Theorem 4.10]{platonov-rapinchuk}) implies that the arithmetic subgroup $\Gamma \subset G^{\mr{ad}}(\QQ)$ is Zariski-dense. By Proposition \ref{svsuperrig} and this Zariski-density, the results of \S \ref{abstract} apply to $\Gamma \subset G^{\mr{ad}}(\QQ)$ (the ``$\rho$" of \textit{loc. cit.}). 
Specifically, we may apply part (1) of Proposition \ref{nearcompatible} (note that all the fields ``$F(\lambda)$" of \textit{loc. cit.} equal $E_{K_0, s}$) to deduce that, enlarging $N$ if necessary, for all $\ell$ the adjoint projection of $\rho_{K_0, s, \ell}$ extends to a homomorphism
\[
\rho^{\mr{ad}}_{K_0, s, \ell} \colon \pi_1(\mc{S}_{K_0, s}[1/\ell], \bar{s}) \to G^{\mr{ad}}(\QQ_{\ell}).
\]
(We do not change the notation, but our new $N$ is what was denoted $N'$ in Proposition \ref{nearcompatible}.) 
To lighten the notation, since $K_0$ and $s$ are fixed, and we only study the $G^{\mr{ad}}$-valued homomorphisms, we will for the remainder of the paper abbreviate
\begin{equation}\label{adjointrep}
\rho_{\ell}=\rho^{\mr{ad}}_{K_0, s, \ell}, \qquad \qquad  \rho_{\ell, v} = \rho^{\mr{ad}}_{K_0, s, \ell}|_{(\mc{S}_{K_0, s})_{\kappa(v)}},
\end{equation}
for $v$ a closed point of $\Spec\,  \mc{O}_{E_{K_0, s}}$.
\begin{cor}\label{svnearcompatible}
For each pair of primes $\ell$ and $\ell'$, and each prime $v$ of $\mc{O}_{E_{K_0, s}}[1/(N\ell \ell')]$, there exists $\tau= \tau(\ell, \ell', v) \in \Aut(G^{\mr{ad}}_{\ol{\QQ}_{\ell'}})$ such that $\rho_{\ell, v}$ and $\tau(\rho_{\ell', v})$ are companions.
\end{cor}
\begin{proof}
This follows from part (2) of Proposition \ref{nearcompatible}. 
\end{proof}
\subsection{Construction of special points}
Our task now is to show that the automorphism $\tau$ in the Corollary is inner, using Proposition \ref{symbreak}. To use \textit{loc. cit.}, we need to show that Assumption \ref{horizontal} holds for the prime $v$. Our technique is to ``break the symmetry" using the horizontal sections of $\mc{S}_{K_0, s}$ given by carefully-chosen special points.

To construct the necessary special points on $\Sh_{K_0}(G, X)$, we will need to construct tori in $G$ with good properties. Recall that an algebraic torus $T$ over a field $k$ is called \textit{irreducible} if it contains no proper subtori (defined over $k$), or equivalently if the $\Gal(\bar{k}/k)$-representation $\rho_T$ on the $\QQ$-linearized character group $X^\bullet(T)_{\QQ}$ is irreducible. The following lemma is a variant of \cite[Theorem 1]{prasad-rapinchuk:irrtori}, which treats the absolutely simple case, and we thank the referee for noting a missing explanation in our earlier version of this argument; we fill it by adapting the argument of \textit{loc. cit.}, but we could alternatively apply \cite[Lemma 3.4]{geyer:approx}.
\begin{lem}\label{lemma:torus-approximation}
Let $H= \prod_{i \in I} H_i$ be an adjoint algebraic group over $\QQ$, where the $H_i$'s are the (adjoint) $\QQ$-simple factors of $H$. Let $S$ be a finite set of places of $\QQ$, and for each $v \in S$ fix a maximal torus $T_v \subset H_{\QQ_v}$ and an open subgroup $K_v \subset H(\QQ_v)$. Then there exists a maximal torus $T \subset H$ such that:
\begin{enumerate}
    \item When we decompose $T= \prod_{i \in I} T_i$ with each $T_i$ a maximal torus in $H_i$, each $T_i$ is irreducible.
    \item For all $v \in S$, $T_{\QQ_v}$ is $K_v$-conjugate to $T_v$.
\end{enumerate}
\end{lem}
\begin{proof}
Since $H$ and each $T_v$ and $K_v$ canonically decompose as products over the indexing set $I$, it clearly suffices to prove the lemma for each $\QQ$-simple factor $H_i$ separately. We thus may and do assume that $H$ is $\QQ$-simple. Let $\mr{Tori}_H$ denote the variety of maximal tori in $H$. By a theorem of Chevalley (\cite[\S 2.4 Theorem 1] {voskresenskii:birat}), $\mr{Tori}_H$ is $\QQ$-rational. We obtain the generic torus (see \cite[\S 2.4.1]{voskresenskii:birat}) $\mc{T}$ over the function field $\QQ(\mr{Tori}_H)$. Let $k/\QQ(\mr{Tori}_H)$ denote its splitting field; then Voskresenskii has proven (\cite[\S 2.4.2 Theorem 2]{voskresenskii:birat}) that $\Gal(k/\QQ(\mr{Tori}_H))$ is isomorphic via the action on $X^\bullet(\mc{T})$ to $W(H, \mc{T}) \rtimes A_0 \subset \Aut(\Phi(H, \mc{T}))$, where $A_0$ denotes the image of the $*$-action of the Galois group. Note that $\mr{Tori}_H$ is a rational variety that satisfies the weak approximation property (\cite[Proposition 7.3]{platonov-rapinchuk}). The former property yields by Hilbert irreducibility a $T_0 \in \mr{Tori}_H(\QQ)$ such that the $\Gal(\ol{\QQ}/\QQ)$-action on $X^\bullet(T_0)$ has image $W(H, T_0) \rtimes A_0$. 

The argument of \cite[Theorem 1(i)]{prasad-rapinchuk:irrtori} finds a finite set $S'$ of finite places disjoint from $S$ such that any $T \in \mr{Tori}_H(\QQ)$ for which $T_{\QQ_v}$ is $H(\QQ_v)$-conjugate to $T_{0, \QQ_v}$ for all $v \in S'$ satisfies: for $g_0 \in H(\ol{\QQ})$ such that $\inn(g_0)(T)=T_0$, inducing isomorphisms $X^{\bullet}(T) \simeq X^{\bullet}(T_0)$ and $\Aut(X^{\bullet}(T)) \simeq \Aut(X^\bullet(T_0))$, we have $\inn(g_0)(\rho_{T}(\Gamma_{\QQ}))= \rho_{T_0}(\Gamma_{\QQ})$ (see Equation (2), \textit{loc. cit.}). This part of the argument uses only that $\rho_{T_0}(\Gamma_{\QQ})$ contains the whole Weyl group, but by the choice of $T_0$ we moreover deduce that $\rho_T(\Gamma_{\QQ})= W(H, T) \rtimes A_0$.

Now, any such $T$ is irreducible. Indeed, there is a finite extension $F/\QQ$ and an absolutely simple adjoint group $H^s$ over $F$ such that $H \cong \Res_{F/\QQ}(H^s)$, and consequently $T$ has the form $\Res_{F/\QQ} (T^s)$ for some maximal torus $T^s$ of $H^s$. Thus $X^\bullet(T)$ is isomorphic as abelian group to $\bigoplus_{\sigma \colon F \to \ol{\QQ}} X^\bullet(T^s_{\sigma})$, with $\rho_T(\Gamma_{\QQ})$ equal to $\left(\prod_{\sigma} W(H^s_{\sigma}, T^s_{\sigma}) \right) \rtimes A_0$. Since each $W(H^s_{\sigma}, T^s_{\sigma})$ acts irreducibly on $X^\bullet(T^s_{\sigma})_{\QQ}$, and $A_0$ acts transitively on the set of factors (since $H$ itself is $\QQ$-simple), the $\Gal(\ol{\QQ}/\QQ)$-action on $X^\bullet(T)_{\QQ}$ is irreducible, as desired. 

We conclude as in \cite[Theorem 1(ii)]{prasad-rapinchuk:irrtori} by applying the weak approximation property of $\mr{Tori}_H$ to find $T \in \mr{Tori}_H(\QQ) \cap \prod_{S \cup S'} U_v$ where for all $v \in S'$, $U_v= \inn(H(\QQ_v)) (T_{0, \QQ_v})$; and for all $v \in S$, $U_v$ is any open subset contained in $\inn(K_v)(T_v)$ (such orbits are open, by \cite[Proposition 3.3 and its Corollaries]{platonov-rapinchuk}). The conclusions of the Lemma now clearly hold for $T$.

\end{proof}

We continue with our construction of special points. We will (given $v$) construct a special point $x \in X$ satisfying several desiderata, including that $\mu_x$, after descent to $\ol{\QQ}$ and base-change to $\ol{\QQ}_p$, is positive with respect to a $\QQ_p$-rational Borel. The following argument incorporates some additional features we require into the argument of \cite[Proposition 1.2.5]{kisin-madapusipera-shin:hondatate} (which is inspired by \cite[5.12]{langlands-rapoport}). Here and for the remainder of the paper, we write $G^{\mr{ad}}= \prod_{i \in I} G_i$ for the decomposition of $G^{\mr{ad}}$ into $\QQ$-simple groups.
\begin{prop}\label{findspecialpoint}
Continue with the fixed connected component $X^+$ of $X$. Fix a prime $p$ such that $G_{\QQ_p}$ is unramified and an embedding $\iota_p \colon \ol{\QQ} \to \ol{\QQ}_p$. Then we can find a maximal torus $T' \subset G$ and a special subdatum $(T', x') \subset (G, X)$ such that:
\begin{enumerate}
    \item $x' \in X^+$. 
    \item When we decompose the image $(T')^{\mr{ad}}= \prod_i T'_i$ of $T'$ in $G^{\mr{ad}}$ into maximal tori $T'_i \subset G_i$, each $T'_i$ is irreducible.
    \item For all $i \neq j$, $T'_i$ is not isomorphic to $T'_j$.
    \item $T'_{\QQ_p}$ is unramified.
    \item The Hodge cocharacter $\mu'$ of $x'$, after descent to $\ol{\QQ}$ and extension along $\iota_p$, defines a cocharacter of $T'_{\ol{\QQ}_p}$ that is positive with respect to a $\QQ_p$-rational Borel subgroup of $G_{\QQ_p}$ containing $T'_{\QQ_p}$.
\end{enumerate}
\end{prop}
\begin{proof}
Fix any $x_{\infty} \in X^+$, corresponding to $h_{\infty} \colon \mathbb{S} \to G_\RR$. Choose a maximal torus $T_{\infty} \subset G_\RR$ containing the image of $h_\infty$. The image $T_{\infty}^{\mr{ad}}$ in $G^\ad_\RR$ is anisotropic since $\inn(h_\infty(i))$ is a Cartan involution of $G_\RR^\ad$. Let $\mu_{\infty} \in X_\bullet(T_\infty)$ be the Hodge cocharacter of $h_\infty$. The  $G(\CC)$-conjugacy class of $\mu_{\infty}$ is independent of $x_{\infty}$ (and of the choice of $X^+$), and we denote it by $\mu_{X, \infty}$; it is defined over $\ol{\QQ}$ and determines a unique $G(\ol{\QQ})$-conjugacy class $\mu_X$ of cocharacters $\mathbb{G}_{m, \ol{\QQ}} \to G_{\ol{\QQ}}$. For any rational prime $p'$, fix an embedding $\iota_{p'} \colon \ol{\QQ} \to \ol{\QQ}_{p'}$ (taking the given $\iota_p$ in the case $p'=p$), inducing an inclusion of the decomposition group $\Gamma_{\QQ_{p'}} \to \Gamma_{\QQ}$, and for $p'=p$ inducing from $\mu_X$ a $G(\ol{\QQ}_p)$-conjugacy class $\mu_{X, p}$ of cocharacters of $G_{\ol{\QQ}_p}$. 

Since $G_{\QQ_p}$ is unramified, $G_{\QQ_p}$ extends to a reductive group scheme $\ul{G}_{\ZZ_p}$ over $\ZZ_p$. Thus there is a $\QQ_p$-rational Borel subgroup $B_p \subset G_{\QQ_p}$ containing an unramified maximal torus $T_p \subset B_p \subset G_{\QQ_p}$. We let $\mu_p \in X_\bullet(T_p)$ denote the $B_p$-positive representative of $\mu_{X, p}$.

For each unordered pair $i \neq j$ of distinct elements of $I$, choose distinct primes $p_{i, j} \neq p$ and maximal tori $T_{i, p_{i, j}} \subset (G_i)_{\QQ_{p_{i, j}}}$ and $T_{j, p_{i, j}} \subset (G_j)_{\QQ_{p_{i, j}}}$ that are not isomorphic as $\QQ_{p_{i, j}}$-tori. (We can choose all the $p_{i, j}$ such that $G$ splits over $\QQ_{p_{i, j}}$, and then $(G_i)_{\QQ_{p_{i, j}}}$ contains a split maximal torus and (by \cite[Theorem 6.21]{platonov-rapinchuk}) an anisotropic maximal torus. If the isomorphism classes of maximal tori in $(G_i)_{\QQ_{p_{i, j}}}$ and $(G_j)_{\QQ_{p_{i, j}}}$ are different, we are obviously done, and otherwise we can take in one a split and in the other an anisotropic torus.) For $k \neq i, j$, let $T_{k, p_{i, j}} \subset (G_k)_{\QQ_{p_{i, j}}}$ be any maximal torus. By Lemma \ref{lemma:torus-approximation}, there is a maximal torus $T \subset G$ {(strictly speaking, the preimage in $G$ of the result of applying Lemma \ref{lemma:torus-approximation} in $G^{\mr{ad}}$)} such that: 
\begin{enumerate}
    \item Writing the image $T^{\mr{ad}}= \prod_i T_i$ of $T$ in $G^{\mr{ad}}$ as a product of maximal tori $T_i \subset G_i$, each $T_i$ is irreducible.
    \item For each unordered pair $i \neq j$ in $I$, $T^{\mr{ad}}_{\QQ_{p_{i, j}}}$ is $G^{\mr{ad}}(\QQ_{p_{i, j}})$-conjugate to (and in particular isomorphic to) $\prod_k T_{k, p_{i, j}}$.
    \item $T_{\RR}= \inn(g_{\infty}) T_{\infty}$ for some $g_{\infty} \in G(\RR)^+$ (the connected component of the identity in the analytic topology);
    \item there exists $g_p \in G(\QQ_p)$ such that $\inn(g_p)(T_p) = T_{\QQ_p}$.  
\end{enumerate}
Now $\inn(g_p)\mu_p \in X_\bullet(T_{\ol{\QQ}_p})$ is positive with respect to the $\QQ_p$-rational Borel $\inn(g_p)B_p \subset G_{\QQ_p}$, and it descends uniquely to a cocharacter $\mu_T \in X_\bullet(T_{\ol{\QQ}})$. The extension $\mu_T \otimes \CC$ is $G(\CC)$-conjugate to $\mu_{\RR}:= \inn(g_{\infty})\mu_{\infty}$, and so these two $T_{\CC}$-valued cocharacters must be conjugate by some $n \in N_{G^{\mr{sc}}}(T^{\mr{sc}})(\CC)$, the normalizer of the preimage $T^{\mr{sc}}$ of $T$ in the simply connected cover $G^{\mr{sc}}$ of $G^{\mr{der}}$: $\inn(n)\mu_{\RR}= \mu_T \otimes \CC$. Since $T^{\mr{sc}}_{\RR}$ is anisotropic, $n \bar{n}^{-1} \in T^{\mr{sc}}(\CC)$.\footnote{It suffices to show that for all $\chi \in X^\bullet(T^{\mr{sc}})$, $\chi(\bar{n}^{-1}t \bar{n})=\chi(n^{-1}tn)$ for all $t \in T^{\mr{sc}}(\CC)$. Complex conjugation $c$ acts by $-1$ on $X^\bullet(T^{\mr{sc}})$ since $T^{\mr{sc}}_{\RR}$ is anisotropic, so 
\[
\chi(n^{-1}t^{-1}n)=(n \cdot \chi)^{-1}(t)= {}^c(n \cdot \chi)(t)=\ol{\chi(n^{-1} \bar{t} n)}=({}^c \chi)(\bar{n}^{-1}t \bar{n})= \chi(\bar{n}^{-1} t^{-1} \bar{n}).
\]}
Let $\alpha_{\infty} \in H^1(\RR, T^{\mr{sc}})$ be the class of the unique cocycle sending complex conjugation $c$ to $n \bar{n}^{-1}$. By \cite[Lemma 1.2.1]{kisin-madapusipera-shin:hondatate} and its proof, there is a class $\alpha \in H^1(\QQ, T^{\mr{sc}})$ such that
\begin{enumerate}
    \item $\alpha|_{\RR}= \alpha_{\infty}$; 
    \item $\alpha|_{\QQ_{p'}}=0$ for all primes $p'$ except for a single auxiliary $p_0 \not \in \{p\} \cup \{p_{i, j}\}_{i, j}$.
\end{enumerate}
Let $i \colon T^{\mr{sc}} \to G^{\mr{sc}}$ be the inclusion. The class $i(\alpha)$ is trivial everywhere locally, as is clear from the construction of $\alpha_{\infty}$ and from the vanishing (\cite[Theorem 6.4]{platonov-rapinchuk}) of $H^1(\QQ_{p_0}, G^{\mr{sc}})$. By the Hasse principle for simply-connected groups (\cite[Theorem 6.6]{platonov-rapinchuk}), $i(\alpha)=0$, i.e., there exists $g \in G^{\mr{sc}}(\ol{\QQ})$ such that $\sigma \mapsto g \sigma(g)^{-1} \in T^{\mr{sc}}(\ol{\QQ})$ is a cocycle representative of $\alpha$.

For all primes $p' \neq p_0$ (including $p'=p$), $\alpha|_{\QQ_{p'}}=0$ shows that there exists $s_{p'} \in T^{\mr{sc}}(\ol{\QQ}_{p'})$ such that $g\sigma(g)^{-1}= s_{p'} \sigma(s_{p'})^{-1}$ for all $\sigma \in \Gamma_{\QQ_{p'}}$, i.e. $s_{p'}^{-1}g$ lies in $G^{\mr{sc}}(\QQ_{p'})$. Similarly, $\alpha|_{\RR}= \alpha_{\infty}$ implies that for some $t \in T^{\mr{sc}}(\CC)$, $g \bar{g}^{-1}= t n \bar{n}^{-1} \ol{t}^{-1}$, i.e. $g^{-1}tn \in G^{\mr{sc}}(\RR)$.

The torus $\inn(g)^{-1}T_{\ol{\QQ}}$ is defined over $\QQ$ (since $T$ is and $g\sigma(g)^{-1} \in T^{\mr{sc}}(\ol{\QQ})$ for all $\sigma \in \Gamma_{\QQ}$); let $T' \subset G$ be the underlying maximal $\QQ$-torus of $\inn(g)^{-1}T_{\ol{\QQ}}$, and let $\mu'= \inn(g)^{-1}\mu_T \in X_{\bullet}(T')$. Then $\inn(n^{-1}t^{-1}g)\mu'= \inn(g_{\infty})\mu_{\infty}$, so (recall $g^{-1}tn \in G^{\mr{sc}}(\RR)$) $\mu'$ is in the same $G^{\mr{sc}}(\RR)$-conjugacy class, hence the same $G(\RR)^+$-conjugacy class, as $\mu_{\infty}$ (for any reductive group over $\RR$, $G^{\mr{sc}}(\RR)$ is connected).  
That is, $\mu'$ corresponds to a point $x'= \inn(g^{-1}tng_\infty)x_\infty$
in the same component $X^+$ of $X$ as $x_{\infty}$.

We now check the remaining desiderata for $(T', x')$. First note that for all $p' \neq p_0$, we have
\[
T'_{\ol{\QQ}_{p'}}= \inn(g^{-1}) T_{\ol{\QQ}_{p'}}= \inn(g_{\QQ_{p'}}^{-1}s_{p'}^{-1})T_{\ol{\QQ}_{p'}}= \inn(g_{\QQ_{p'}}^{-1})T_{\ol{\QQ}_{p'}},
\]
for $g_{\QQ_{p'}} = s_{p'}^{-1}g \in G^{\mr{sc}}(\QQ_{p'})$ (recall $s_{p'} \in T^{\mr{sc}}(\ol{\QQ}_{p'})$). Thus $\inn(g_{\QQ_{p'}}^{-1})T_{\QQ_{p'}}$ and $T'_{\QQ_{p'}}$ are $\QQ_{p'}$-subtori of $G_{\QQ_{p'}}$ that are equal in $G_{\ol{\QQ}_{p'}}$, and hence are already equal over $\QQ_{p'}$: $\inn(g_{\QQ_{p'}}^{-1})T_{\QQ_{p'}}= T'_{\QQ_{p'}}$. In particular this holds for $p'=p$, so $T'_{\QQ_p}$ is also unramified (since $g_p \in G(\QQ_p)$) and we see $\mu'= \inn(g^{-1})\mu_T= \inn(g_{\QQ_p}^{-1}) \mu_T= \inn(g_{\QQ_p}^{-1} g_p) \mu_p$; thus $\mu'$ is positive with respect to the $\QQ_p$-rational Borel subgroup $B'= \inn(g_{\QQ_p}^{-1}g_p)B_p$ containing $T'_{\QQ_p}$. We also conclude that for all distinct $i, j \in I$, $(T'_i)_{\QQ_{p_{i, j}}}$ is isomorphic to $(T_i)_{\QQ_{p_{i, j}}}$ and thus is not isomorphic to $(T_j)_{\QQ_{p_{i, j}}} \cong (T'_j)_{\QQ_{p_{i, j}}}$. In particular $T'_i$ is not isomorphic to $T'_j$ for distinct $i, j$.

It remains only to check that each $\QQ$-torus $T'_i$ is irreducible. We have seen that for all primes $p' \neq p_0$, $T_{\QQ_{p'}}$ and $T'_{\QQ_{p'}}$ are $G(\QQ_{p'})$-conjugate. In particular, for all such $p'$ and for all $i \in I$, $(T'_i)_{\QQ_{p'}}$ is isomorphic to $(T_i)_{\QQ_{p'}}$. The Galois representations $\rho_{T_i}, \rho_{T'_i} \colon \Gamma_{\QQ} \to \mr{GL}_{n_i}(\QQ)$ (of some dimensions $n_i$) on the $\QQ$-vector spaces $X^{\bullet}(T_i)_{\QQ}$ and $X^{\bullet}(T'_i)_{\QQ}$ are therefore isomorphic after restriction to each decomposition group $\Gamma_{\QQ_{p'}}$, $p' \neq p_0$. By the Brauer-Nesbitt theorem and the \v{C}ebotarev density theorem, $\rho_{T_i}$ and $\rho_{T'_i}$ are isomorphic for all $i$;\footnote{Note that Brauer-Nesbitt shows they are isomorphic as representations to $\mr{GL}_n(\QQ)$, not merely as $\mr{GL}_n(\ol{\QQ})$-representations.} thus the irreducibility of $\rho_{T_i}$ implies that of $\rho_{T'_i}$.
\end{proof}

\begin{rmk}
    In the Siegel moduli case, conditions (4) and (5) in the Proposition imply that the associated CM abelian variety is ordinary at a suitable place above $p$. More generally, by \cite[Proposition 9.10]{bakker-shankar-tsimerman:canonicalmodels}, they imply the special subdatum $(T', x')$ is $\mu$-ordinary in the sense of \cite[Definition 9.1]{bakker-shankar-tsimerman:canonicalmodels} (generalizing a definition implicit in \cite{rapoport-richartz} and explicit in \cite{wedhorn:ordinary}). The proof of non-emptiness of the $\mu$-ordinary locus in general (\cite[Theorem 9.2]{bakker-shankar-tsimerman:canonicalmodels}) relies on an argument closely-related to that of Proposition \ref{specialgalois} to produce special points satisfying conditions (1), (4), (5), and a weakening of (2) (that the Galois group of the $T_i'$ contain the Weyl group, as in \cite[Proposition 5.11]{chai-oort:AVisogJac}).
\end{rmk}

We next make the output of Lemma \ref{specialgalois} more explicit in the form needed for our application. 
\begin{lem}\label{infiniteorder}
Let $(T_x, x) \subset (G, X)$ be a special subdatum with $x \in X^+$, and for $a \in G(\A_f)$, let $s=[x, a] \in \Sh(\CC)$ be the associated special point. Let $v$ be a prime of $\mc{O}_{E(sK_0)}[1/N])$ above a rational prime $p$, and assume:
\begin{enumerate}
\item When we write $T_x^{\mr{ad}}= \prod_{i \in I} T_{x, i}$ as a product of maximal tori $T_{x, i} \subset G_i$, each $T_{x, i}$ is irreducible, and for all $i \neq j$, $T_{x, i}$ is not isomorphic to $T_{x, j}$.
\item $G_{\QQ_p}$ is unramified, $T_{x, \QQ_p}$ is unramified, and when we fix an embedding $\iota \colon \ol{\QQ} \to \ol{\QQ}_p$ inducing the place $v$ and a commutative diagram of field embeddings
\[
\begin{tikzcd}
\ol{\QQ} \ar[r, hookrightarrow, "\iota"] & \ol{\QQ}_p \\
E(sK_0) \ar[r, hookrightarrow] \ar[u, phantom, sloped, "\subset"] & E(sK_0)_v, \ar[u, phantom, sloped, "\subset"]
\end{tikzcd}
\]
$\mu_{x, \iota}$ is positive with respect to a $\QQ_p$-rational Borel subgroup of $G_{\QQ_p}$ containing $T_{x, \QQ_p}$. (Here $\mu_{x, \iota}$ denotes the extension of $\mu_x$ along $\iota$ to an element of $X_{\bullet}(T_{x, \ol{\QQ}_p})$.)

\item $K_0$ contains $\underline{G}_{\ZZ_p}(\ZZ_p)$, and ($a_p$ being the $p$-component of $a \in G(\A_f)$) $a_p$ is in $\underline{G}_{\ZZ_p}(\ZZ_p)$.
\end{enumerate}
Let $\ell \neq p$ be a distinct prime. Then setting $\rho_{\ell, sK_0}= \rho_{\ell}|_{sK_0}$, the projection to each component $G_i(\QQ_{\ell})$ of the element $\rho_{\ell, sK_0}(\Frob_v)$ has infinite order.
\end{lem}
\begin{proof}
In this proof, we observe Convention \ref{reclazy}. By definition, the Mumford-Tate group $M_x$ of $h_x$ is contained in $T_x$. 
For all $i \in I$, $T_{x, i}$ contains no proper $\QQ$-subtorus, and (by definition of a Shimura datum) $h_x$ has non-trivial projection to each $G_{i, \RR}$. Thus $M_x$ surjects onto each $T_{x, i}$, so the canonical composition
\[
X^\bullet(T_{x, i}) \hookrightarrow X^\bullet(T_x^{\mr{ad}}) \twoheadrightarrow X^{\bullet}(M_x^{\mr{ad}})
\]
is injective for each $i$. Since the $\QQ[\Gamma_{\QQ}]$-modules $X^{\bullet}(T_{x, i})_{\QQ}$ are irreducible and non-isomorphic, $\bigoplus_i X^{\bullet}(T_{x, i})$ injects into (and as we know surjects onto) $X^{\bullet}(M_x^{\mr{ad}})$ via the direct sum of canonical maps, hence $M_x^{\mr{ad}}= T_x^{\mr{ad}}$. Consequently, $M_x$ contains $T_x \cap G^{\mr{der}}$.

Recall from Lemma \ref{specialgalois} (the notation of which we maintain) that to compute the Galois action we write $qx=x$, $qak_0(\sigma)=r_x(\alpha)a$ for $\alpha \in \A_{E(sK_0)}^\times$ mapping to $\sigma \in \Gamma_{E(sK_0)}^{\mr{ab}}$ under the reciprocity map. Lemma 2.2 of \cite{ullmo-yafaev:generalized} and the above observation about $M_x$ shows that $q \in \mr{Cent}_G(M_x)(\QQ)= T_x(\QQ)$. Thus $r_x(\alpha)=q \cdot (ak_0(\sigma)a^{-1}) \in T_x(\QQ) \cdot (aK_0a^{-1} \cap T_x(\A_f))$. The proof of \cite[Lemma 2.1]{ullmo-yafaev:generalized} shows that $T_x(\QQ) \cap aK_0a^{-1}=\{1\}$, using that $T_x= \mr{Cent}_G(M_x)$ and that $K_0$ is neat. In particular, the decomposition $r_x(\alpha)=q ak_0(\sigma)a^{-1}$ as above is unique (given $\alpha$; choosing a different $\alpha$ mapping to $\sigma$ will change $q \in T_x(\QQ)$).

By definition of the integral model $\mc{S}_{K_0, s}$ and \cite[Theorem 4.1]{pila-shankar-tsimerman:andre-oort}, $sK_0=[x, a]_{K_0}$ extends to a point of $\mc{S}_{K_0, s}(\mc{O}_{E(sK_0)}[1/N])$,\footnote{Note that $\mc{S}_{K_0, s}(\mc{O}_{E(sK_0)}[1/N])= \mc{S}_{K_0, s}(E(sK_0)) \cap \mc{S}_{K_0, s}(\overline{\ZZ}[1/N])$ inside $\mc{S}_{K_0, s}(\ol{\QQ})$, the inclusion making sense because $\mc{S}_{K_0, s}$ is separated. Indeed, for any finite Galois extension $L/E(sK_0)$, $\Spec(\mc{O}_L[1/N]) \to \Spec( \mc{O}_{E(sK_0)}[1/N])$ is a categorical quotient by the group $\mr{Gal}(L/E(sK_0))$ of automorphisms, so $\mc{S}_{K_0, s}(\mc{O}_L[1/N])^{\mr{Gal}(L/E(sK_0))}= \mc{S}_{K_0, s}(\mc{O}_{E(sK_0)}[1/N])$.} so for $v \nmid N\ell$, $\rho_{\ell, sK_0}$ is unramified at $v$. For a uniformizer $\varpi_v \in E(sK_0)_v$ mapping to a Frobenius element $\Frob_v \in \Gamma_{E(sK_0)}^{\mr{ab}}$ we have
\[
r_x(\varpi_v)= q a k_0(\Frob_v) a^{-1}
\]
for a unique $q \in T_x(\QQ)$. Considering this identity in the $\ell$ component we see 
\begin{equation}\label{CMfrob}
a_{\ell}^{-1} q^{-1} a_{\ell}= k_{0, \ell}(\Frob_v).
\end{equation}
Considering the identity in the $p$ component, where $v \vert p$, we see
\begin{equation}\label{locallyalg}
q^{-1} N_{E(sK_0)_v/\QQ_p}(\mu_x(\varpi_v)))= a_p k_{0, p}(\Frob_v) a_p^{-1}.
\end{equation}
The right-hand side lies in the \textit{compact} open subgroup $T_x(\QQ_p) \cap a_p K_{0, p} a_p^{-1}$ of $T_x(\QQ_p)$, and in particular for all cocharacters $\chi \in X^{\bullet}(T_x)$, $\chi(q^{-1} N_{E(sK_0)/\QQ_p}(\mu_x(\varpi_v)))$ has trivial $p$-adic valuation. For each $i \in I$, we project Equation \eqref{locallyalg} to the $G_i$ component of $G^{\mr{ad}}$, and by looking at $p$-adic valuations we see that the image in $T_{x, i}(\QQ)$ of $q \in T_x(\QQ)$ has infinite order---and thus (Equation \eqref{CMfrob}) the image in $G_i(\QQ_{\ell})$ of $k_{0, \ell}(\Frob_v)$ has infinite order---if for some root $\chi$ of $(G_i, T_{x, i})$ (or non-trivial character of $T_{x, i}$, $\chi(N_{E(sK_0)_v/\QQ_p}(\mu_x(\varpi_v)))$ (which depends only on the projection of $\mu_x(\varpi_v)$ to $G_i$) has non-zero $p$-adic valuation.

We will now use our additional assumptions (2) and (3). The torus $T_{x, \QQ_p}$ splits over an unramified extension of $\QQ_p$, so $\mu_x$ is defined over such an extension, and $E(x)/\QQ$ is unramified over $p$. In fact $E(sK_0)/E(x)$ is also unramified at all primes above $p$. Indeed, the canonical model characterization (Equation \eqref{reciprocity}) applied to $\alpha \in \mc{O}_{E(x)_v}^\times$, for $v \vert p$, shows that the inertia group $I_{E(x)_v}$ fixes $sK_0 = [x, a]_{K_0}$ once the (unique) maximal bounded (or compact) subgroup $T_x(\QQ_p)^0$ of $T_x(\QQ_p)$ is contained in $a_p K_0 a_p^{-1}$. By \cite[Lemma A.5.2]{getz-hahn:automorphicbook}, $\underline{G}_{\ZZ_p}(\ZZ_p) \cap T_x(\QQ_p)$ is equal to $T_x(\QQ_p)^0$, which is thus under our assumptions contained in $a_p K_0 a_p^{-1}$.

For $p$ as in the previous paragraph, we may take $\varpi_v=p$ to be a uniformizer of $E(sK_0)_v$, for $v \vert p$. To show that $k_{0, \ell}(\Frob_v)$ has infinite order in $G_i(\QQ_{\ell})$, we compute
\[
N_{E(sK_0)_v/\QQ_p}(\mu_x(p))= \prod_{\sigma \colon E(sK_0)_v \to \ol{\QQ}_p} \sigma(\mu_x(p))= \prod_{\sigma \colon E(sK_0)_v \to \ol{\QQ}_p} {}^{\sigma}\mu_x(\sigma(p))= (\sum_{\sigma \colon E(sK_0)_v \to \ol{\QQ}_p} {}^{\sigma}\mu_x)(p).
\]
(We write ${}^\sigma \mu_x$ for the action on the cocharacter $\mu_x$ and $\sigma(\mu_x(p))$ for the action on the value $\mu_x(p)$.) By construction $\mu_{x, \iota}$ is positive with respect to a $\QQ_p$-rational Borel subgroup $B_x$ of $G_{\QQ_p}$, and thus its projection to $T_{x, i}$ is positive with respect to each projection $B_{x, i} \subset G_i$ (a $\QQ_p$-rational Borel) of $B_x$: note that ``positive" here means strictly positive, since $\mu_x$ is non-trivial in each $G_i$. Our embedding $\iota$ induces one of the embeddings $\sigma$ in the above product, so the $G_i$-projection of $\sum_{\sigma \colon E(sK_0)_v \to \ol{\QQ}_p} {}^\sigma \mu_x$ is also positive with respect to the $\QQ_p$-rational Borel $B_{x, i}$. In particular, for some $B_{x, i}$-positive root $\chi \in X^\bullet(T_{x, i})$, $\chi \circ \left( \sum_{\sigma \colon E(sK_0)_v \to \ol{\QQ}_p} {}^\sigma \mu_x \right) \in \Hom(\mathbb{G}_m, \mathbb{G}_m)$ is $z \mapsto z^n$ for some $n>0$, and 
\[
\chi(N_{E(sK_0)_v/\QQ_p}(\mu_x(p)))= p^n
\]
has positive $p$-adic valuation. We conclude that for each $i$ the projection to $G_i(\QQ_{\ell})$ of $k_{0, \ell}(\Frob_v)$ has infinite order.
\end{proof}
We can now assemble the ingredients into the proof of our main theorem:
\begin{thm}\label{mainthm}
Let $(G, X)$ be a Shimura datum such that $Z_G(\QQ)$ is a discrete subgroup of $Z_G(\A_f)$, and thus for every neat compact open subgroup $K_0 \subset G(\A_f)$ and $s \in \Sh(G, X)(\CC)$, we have the canonical $\ell$-adic local systems
\[
\rho_{K_0, s, \ell} \colon \pi_1(S_{K_0, s}, \bar{s}) \to K_{0, \ell}
\]
defined in \S \ref{localsysdef} with adjoint projection denoted $\rho_{\ell}$.

Assume:
\begin{equation}
    \text{$(\ast)$ For every $\QQ$-simple factor $H$ of $G^{\mr{ad}}$, $\mr{rk}_{\RR} (H_{\RR}) \geq 2$.} 
\end{equation}
Consequently, as described in \S \ref{integralmodels}, for all $\ell$, $\rho_{\ell}$ extends to the integral model $\mc{S}_{K_0, s}[1/\ell] \to \Spec(\mc{O}_{E_{K_0, s}}[1/N\ell])$,
\[
\rho_{\ell} \colon \pi_1(\mc{S}_{K_0, s}[1/\ell], \bar{s}) \to G^{\mr{ad}}(\QQ_{\ell}).
\]

Then there exists an integer $\widetilde{N}$ such that for all primes $\ell, \ell'$, $\rho_{\ell}$ and $\rho_{\ell'}$ are companions on $\mc{S}_{K_0, s}[1/\widetilde{N} \ell \ell']$ in the sense that for all closed points $x \in \mc{S}_{K_0, s}[1/\widetilde{N} \ell \ell']$, the conjugacy classes of $\rho_{\ell}(\Frob_x)$ and $\rho_{\ell'}(\Frob_x)$ are $\QQ$-rational and are represented by the same element of $[G^{\mr{ad}}\git G^{\mr{ad}}](\QQ)$. 
\end{thm}
\begin{proof}
$G^{\mr{ad}}$ is a product of $\QQ$-simple factors, and it suffices to prove the asserted compatibility for each $\QQ$-simple factor. Fix a set of representatives $\mc{C} \subset G(\A_f)$ of $G(\QQ) \backslash G(\A_f)/K_0$, and fix a smooth reductive model $\underline{G} \to \Spec(\ZZ[1/N_0])$ such that for all $p \nmid N_0$, $G_{\QQ_p}$ is unramified. Let $N_1$ be the product of all primes $p$ such that for one of the finitely many elements $a \in \mc{C}$, the $p$-component $a_p$ does not lie in $\underline{G}_{\ZZ_p}(\ZZ_p)$, and let $N_2$ be the product of all primes $p$ such that $K_0$ does not contain $\underline{G}_{\ZZ_p}(\ZZ_p)$. Finally set $\widetilde{N}=N N_0 N_1 N_2$.  

Fix a connected component $X^+$ of $X$. Each $a \in \mc{C}$ then determines as in Equation \eqref{components} a geometrically connected component of $\Sh_{K_0}$, and for any complex geometric point $\bar{s}_a= [x_a, a]$ ($x_a \in X^+$) of each such component, Proposition \ref{svsuperrig} and the discussion following produces the local systems $\rho_{\ell}$ on $\mc{S}_{K_0, s_a}[1/\ell]$. We will check the claimed compatibility after restriction to $\mc{S}_{K_0, s_a}[1/\widetilde{N} \ell \ell']$. Changing the base-point conjugates $\rho_{\ell}$, and in the next paragraph we will specify convenient base-points for each component.

Fix a prime $v$ of $\mc{O}_{E_{K_0, s}}[1/\widetilde{N}\ell \ell']$, above some rational prime $p$, and fix an embedding $\iota_p \colon \ol{\QQ} \to \ol{\QQ}_p$ inducing $v$, as in Lemma \ref{infiniteorder}. It suffices to prove, for each such $v$ independently, that the companion property holds for the restrictions $\rho_{\ell, v}$ and $\rho_{\ell', v}$ to $\mc{S}_{\kappa(v)}:= \mc{S}_{K_0, s} \times_{\Spec(\mc{O}_{E_{K_0, s}})} \Spec(\kappa(v))$. Choose a special subdatum $(T_x, x) \subset (G, X)$ as in Proposition \ref{findspecialpoint}, so $x \in X^+$, $T_x^{\mr{ad}}= \prod_{i \in I} T_{x, i}$ is a product of maximal tori $T_{x, i} \subset G_i$ with the $T_{x, i}$ all irreducible and non-isomorphic as $\QQ$-tori, and the Hodge cocharacter $\mu_x \colon \mathbb{G}_{m, \ol{\QQ}} \to T_{x, \ol{\QQ}}$ (the canonical $\ol{\QQ} \subset \CC$ descent of the complex cocharacter) defines $\mu_{x, \iota_p} \in X_{\bullet}(T_{x, \QQ_p})$ that is positive with respect to a $\QQ_p$-rational Borel subgroup of $G_{\QQ_p}$ containing the unramified torus $T_{x, \QQ_p}$. We now specify our base-point $s_a$ to be $\bar{s}= [x, a] \in \Sh(\CC)$, inducing also $sK_0 \in \mc{S}_{K_0, s}(\mc{O}_{E(sK_0)}[1/N])$ (see the proof of Lemma \ref{infiniteorder} for the integrality). Let $w \vert v$ be the place of $\mc{O}_{E(sK_0)}$ above $v$ such that $\iota_p$ also induces $w$, and consider the restriction of our local system $\rho_{\ell}$ along the map $s_w$ in the commutative diagram (the left vertical arrow arising by specializing $sK_0$)
\begin{center}
    \begin{tikzcd}
        \Spec(\kappa(w)) \arrow[d] \arrow[rr] \arrow[rrd, "s_w"] & & \Spec(\mc{O}_{E(sK_0)}[1/\widetilde{N}\ell \ell']) \arrow[d, "sK_0"] \\
        \mc{S}_{\kappa(w)} \arrow[r] & \mc{S}_{\kappa(v)} \arrow[r] & \mc{S}_{K_0, s}[1/\widetilde{N} \ell \ell'].
    \end{tikzcd}
\end{center}
The Galois representations $\rho_{\ell}|_{sK_0}$ and $\rho_{\ell'}|_{sK_0}$ are compatible, as is proven easily using Equation \eqref{reciprocity}: see Equation \eqref{CMfrob} (and for a classical example see \cite[\S II.2.8]{serre:ladic}). By Proposition \ref{symbreak}, to show the compatibility of $\rho_{\ell, v}$ and $\rho_{\ell', v}$ it suffices to verify that Assumption \ref{horizontal} holds, which, given the compatibility after restriction to $sK_0$, amounts to checking that $\rho_{\ell}|_{sK_0}(\Frob_w)$ is not fixed by any outer automorphism of $G^{\mr{ad}}_{\ol{\QQ}_{\ell}}$. 

By Lemma \ref{infiniteorder} and our requirement $v \nmid \widetilde{N}\ell \ell'$, $\rho_{\ell}|_{sK_0}(\Frob_w) \in K_0 \cap (a_{\ell}^{-1} T_x^{\mr{ad}}(\QQ_{\ell}) a_{\ell})$ has infinite order in each $G_i(\QQ_{\ell})$. Define the representation $r_{\ell}$ by 
\[
r_{\ell}=a_{\ell} \rho_{\ell}|_{sK_0} a_{\ell}^{-1} \colon \pi_1(E(sK_0), \bar{s}) \to T_x^{\mr{ad}}(\QQ_{\ell}).
\]

Note that, per Equation \eqref{CMfrob} (where what is now denoted $w$ was denoted $v$), $r_{\ell}(\Frob_w)$ is in $T_x^{\mr{ad}}(\QQ)$. If $r_\ell(\Frob_w)$ were fixed by an automorphism $\tau$ of $G^{\mr{ad}}_{\ol{\QQ}_{\ell}}$, then $\tau$ would also fix point-wise the $\QQ$-Zariski closure $\ol{\langle r_{\ell}(\Frob_w) \rangle}^{\mr{Zar}}$ of the subgroup generated by $r_{\ell}(\Frob_w) \in T_x^{\mr{ad}}(\QQ)$. In each $G_i$-projection, $r_{\ell}(\Frob_w)$ has infinite order, so $\ol{\langle r_{\ell}(\Frob_w) \rangle}^{\mr{Zar}}$ must surject onto each of the irreducible tori $T_{x, i}$.\footnote{We have used in this argument that $\ol{\langle r_{\ell}(\Frob_w) \rangle}^{\mr{Zar}}$ is the corresponding Zariski-closure in $G^{\mr{ad}}$ (over $\QQ$), and that its base-change to $\ol{\QQ}_{\ell}$ is the Zariski-closure in $G^{\mr{ad}}_{\ol{\QQ}_{\ell}}$. Indeed, we only know that $\tau$ is an automorphism of $G^{\mr{ad}}_{\ol{\QQ}_{\ell}}$, so we first conclude that $\tau$ fixes the Zariski-closure (over $\ol{\QQ}_{\ell}$) in $G^{\mr{ad}}_{\ol{\QQ}_{\ell}}$ of the subgroup $\langle r_{\ell}(\mr{Frob}_w) \rangle$, and then apply this base-change property of Zariski-closure.} Arguing as in the first paragraph of Lemma \ref{infiniteorder} using that the tori $T_{x, i}$ are pairwise non-isomorphic, we deduce that $\ol{\langle r_{\ell}(\Frob_w) \rangle}^{\mr{Zar}}$ is equal to $T_x^{\mr{ad}}$. Thus $\tau$ is an automorphism of $G^\mr{ad}_{\ol{\QQ}_\ell}$ inducing the identity on $T^{\mr{ad}}_{x, \ol{\QQ}_\ell}$. It consequently induces the identity automorphism of the root datum $R(G^{\mr{ad}}_{\ol{\QQ}_\ell}, T^{\mr{ad}}_{x, \ol{\QQ}_\ell})$ and hence by the isomorphism theorem for reductive groups (\cite[Theorem 1.3.15]{conrad:luminy}) is conjugation by an element of $T^{\mr{ad}}_x(\ol{\QQ}_\ell)$. In particular $\tau$ is inner. It immediately follows that $\rho_{\ell}|_{sK_0}(\Frob_w)$ is likewise fixed by no non-trivial outer automorphism of $G^{\mr{ad}}_{\ol{\QQ}_{\ell}}$, so we have verified the second hypothesis of Assumption \ref{horizontal} and conclude by Proposition \ref{symbreak} that for all $v \nmid \widetilde{N}\ell \ell'$, $\rho_{\ell}|_{\mc{S}_{\kappa(v)}}$ and $\rho_{\ell'}|_{\mc{S}_{\kappa(v)}}$ are companions.

Thus for each $v \nmid \widetilde{N}$ and closed point $x \in \mc{S}_{\kappa(v)}$ we obtain a semisimple conjugacy class $\alpha_x \in [G^{\mr{ad}}\git G^{\mr{ad}}](\ol{\QQ})$ given by the common underlying semisimple class of the $\rho_{\ell}(\mr{Frob}_x)$ (whenever $\ell$ is not below $v$), and such that (again when $v \nmid \widetilde{N}\ell$) $\alpha_x$ lies in $[G^{\mr{ad}}\git G^{\mr{ad}}](\QQ_{\ell})$ for any embedding $\ol{\QQ} \to \ol{\QQ}_{\ell}$. Since $\alpha_x$ is defined over some finite extension $E$ of $\QQ$ and is invariant under decomposition groups above $\ell$ for almost all $\ell$, the \v{C}ebotarev density theorem implies $\alpha_x$ lies in $[G^{\mr{ad}}\git G^{\mr{ad}}](\QQ)$. (This rationality argument appears in \cite[Theorem 6.1.4]{kisin-zhou}.)
\end{proof}
\begin{rmk}
    It is perhaps clarifying to add the following to the discussion of $r_{\ell}$ in the last proof. The proof of Lemma \ref{infiniteorder} up to Equation \eqref{CMfrob} applies for any prime $w' \in \Spec(\mc{O}_{E(sK_0)}[1/N \ell])$ to show that for all such $w'$, $r_{\ell}(\Frob_{w'}) \in T_x^{\mr{ad}}(\QQ)$ (the later assertions of the lemma rely on the other assumptions). Thus $r_{\ell}$ is an abelian $\QQ$-rational $\ell$-adic representation, necessarily Hodge-Tate (as verified by direct computation using Equation \eqref{locallyalg} or by the theorems of Serre and Waldschmidt), in the sense of Serre's theory (\cite{serre:ladic}). We check that it arises from a (unique) $\QQ$-algebraic representation $r \colon \mathfrak{S}=\mathfrak{S}_{E(sK_0), \mf{m}} \to T_x^{\mr{ad}}$ of the Serre group $\mathfrak{S}_{E(sK_0), \mf{m}}$ for the number field $E(sK_0)$ and some modulus $\mf{m}$: this is the diagonalizable $\QQ$-group denoted $S_{\mf{m}}$ in \cite[\S II-2]{serre:ladic}, with the number field $K$ of \textit{loc. cit.} our $E(sK_0)$). Indeed, for every $\QQ$-rational representation $s_V \colon T_x^{\mr{ad}} \to \Aut(V)$ ($V$ a finite-dimensional $\QQ$-vector space), there is an associated abelian $\ell$-adic $\QQ$-rational representation $s_V \circ r_{\ell}$ of modulus $\mf{m}$, which by \cite[III.2.3 Theorem 2]{serre:ladic} arises from a unique homomorphism $\mathfrak{S} \to \Aut(V)$. We obtain a tensor functor $\mr{Rep}_{\QQ}(T_x^{\mr{ad}}) \to \mr{Rep}_{\QQ}(\mf{S})$, which induces the desired homomorphism $r \colon \mf{S} \to T_x^{\mr{ad}}$.\footnote{The cited theorem shows that there is a $\QQ$-structure $V_0$ on $V_{\QQ_{\ell}}$ such that $s_V \circ r_{\ell}$ arises from a homomorphism $\mf{S} \to \Aut_{V_0}$; but one checks that one can take $V_0=V$ itself, and that the resulting $\QQ$-group homomorphism is unique once the $\QQ$-structure is specified. We then obtain a tensor functor using the uniqueness. Alternatively, here is a more concrete argument to construct the homomorphism $r$. We first define a homomorphism $\mathfrak{S}_{\ol{\QQ}} \to T^{\mr{ad}}_{x, \ol{\QQ}}$ by defining a group homomorphism $X^\bullet(T^{\mr{ad}}_x) \to X^{\bullet}(\mf{S})$ as follows. Letting $L$ be a number field, Galois over $\QQ$, splitting $T_x^{\mr{ad}}$, any $\chi \in X^{\bullet}(T_x^{\mr{ad}})$ is defined over $L$, and for any place $\lambda \vert \ell$ of $L$ we have $\chi_{\lambda} := \chi \circ r_{\ell} \colon \pi_1(E(sK_0)) \to L_{\lambda}^\times$, with the property that for all $v \nmid \mf{m}$, $\Frob_v$ maps to an element of $L^\times$. By Serre's theory (\cite[III.2.3-III.2.4]{serre:ladic}), there is a unique character $r(\chi) \colon \mf{S}_L \to \mathbb{G}_{m, L}$ such that $\chi_{\lambda}$ equals the composite $\pi_1(E(sK_0)) \xrightarrow{\varepsilon_{\ell}} \mf{S}(\QQ_{\ell}) \xrightarrow{r(\chi)} L_{\lambda}^\times$, where $\varepsilon_{\ell}$ is the canonical homomorphism of \cite[\S II.2.3]{serre:ladic}. The assignment $\chi \mapsto r(\chi)$ is clearly a group homomorphism. Thus we obtain $r \colon \mf{S}_L \to T_{x, L}^{\mr{ad}}$ such that for all $\ell$, $r_{\ell}$ is the composite 
\[
\pi_1(E(sK_0)) \xrightarrow{\varepsilon_{\ell}} \mf{S}(\QQ_{\ell}) \to \mf{S}(L_{\lambda}) \xrightarrow{r} T_x^{\mr{ad}}(L_{\lambda}).
\]

To show that $r$ is defined over $\QQ$, we observe that for all $\ell$ and $\lambda \vert \ell$, on $L_{\lambda}$-points $r$ maps to $\mf{S}(\QQ_{\ell})$ to $T^{\mr{ad}}_x(\QQ_{\ell})$: indeed the image of $\varepsilon_{\ell}$ is Zariski-dense in $\mf{S}_{\QQ_{\ell}}$ (\cite[II.2.3, Remark]{serre:ladic}), and we know that each $r_{\ell}$ lands in $T_x^{\mr{ad}}(\QQ_{\ell})$. This implies $r$ is $\mr{Gal}(L_{\lambda}/\QQ_{\ell})$-invariant for all $\ell, \lambda$, and thus (by \v{C}ebotarev's theorem) $r$ is $\mr{Gal}(L/\QQ)$-invariant. Being a morphism of groups of multiplicative type, it is defined over $\QQ$.}  
For our specified prime $w$, $r(\varpi_w)$ ($\varpi_w$ a uniformizer at $w$) is the element $r_{\ell}(\Frob_w) \in T_x^{\mr{ad}}(\QQ)$, and so the argument in Theorem \ref{mainthm} shows $r \colon \mf{S} \to T_x^{\mr{ad}}$ is surjective. 
\end{rmk}

\begin{cor}\label{specialthm}
Continue in the setting of Theorem \ref{mainthm}. Let $F \subset \ol{\QQ}$ be a number field containing the reflex field $E(G, X)$, and let $y \in \mr{Sh}_{K_0}(F)$ be an $F$-rational point. There exists a geometrically-connected component $S_{K_0, s}$ as in the theorem whose field of definition $E_{K_0, s}$ is contained in $F$ such that $y$ lies in $S_{K_0, s}(F)$, and thus in $\mc{S}_{K_0, s}(\mc{O}_F[1/N(y)]$ for some integer $N(y)$ divisible by the integer $N$ of Theorem \ref{mainthm}.

For all $\ell$, we consider the restrictions 
\[
\rho_{\ell, y} \colon \pi_1(\Spec(\mc{O}_F[1/(N(y) \ell)]), \bar{s}) \to G^{\mr{ad}}(\QQ_{\ell})
\]
of $\rho_{\ell}$ along $y$. Then for all finite places $v \nmid \widetilde{N}N(y)\ell$ of $F$, $\rho_{\ell, y}(\mr{Frob}_v)$ has underlying semisimple conjugacy class that is $\QQ$-rational and independent of $\ell$, as an element of $[G^{\mr{ad}}\git G^{\mr{ad}}](\QQ)$.
\end{cor}
\begin{proof}
The composite $\Spec(\ol{\QQ}) \to \Spec(F) \xrightarrow{y} \mr{Sh}_{K_0}$ defines an element of $\mr{Sh}_{K_0, \ol{\QQ}}(\ol{\QQ})= \bigsqcup S_{K_0, s, \ol{\QQ}}(\ol{\QQ})$, where the disjoint union is taken over $\pi_0(\mr{Sh}_{K_0, \ol{\QQ}})$. Thus $y$ hits exactly one geometrically-connected component $S_{K_0, s}$ and is $\mr{Gal}(\ol{\QQ}/F)$-invariant; it follows that $F$ contains $E_{K_0, s}$. The conclusion of the Corollary then follows readily from Theorem \ref{mainthm}. 
\end{proof}
\begin{rmk}
For any prime $\ell$, we can find $F$ and $y \in \mr{Sh}_{K_0}(F)$ as in Corollary \ref{specialthm} such that $\rho_{\ell, y}$ has Zariski-dense, and even open, image in $G^{\mr{ad}}(\QQ_{\ell})$. This follows as in \cite[\S 10.6]{serre:mordell-weil} from Hilbert irreducibility and a Frattini argument. We have not worked out whether one can then show the expected statement that for all $\ell'$, $\rho_{\ell', y}$ also has Zariski-dense image in $G^{\mr{ad}}(\QQ_{\ell'})$. 
\end{rmk}

\begin{rmk}\label{non-adjoint} The basic obstruction to proving compatibility of the $G(\QQ_{\ell})$-local systems on superrigid Shimura varieties arises from the abstract results of \S \ref{abstract}. In Proposition \ref{nearcompatible}, and following its notation except that we allow general $G$ as in \S \ref{svsection}, the primary difficulty in showing the $\lambda$-adic local systems extend to $X_{\mc{O}_F[1/N'\ell]}$ for some $N'$ independent of $\ell$ without the requirement that $G$ is adjoint lies in non-uniqueness of the descents to $X_F$: for varying $\lambda$, any given descents can be twisted by (differently-ramified) characters $\pi_1(X_F) \to Z_{G}(\ol{\QQ}_{\lambda})$. Achieving an analogue of Part (2) of Proposition \ref{nearcompatible} in the Shimura variety setting, however, is additionally obstructed in two ways. First, when we take companions in the $v$-fiber and pull back by tame specialization to the topological fundamental group representation, Margulis superrigidity yields a weaker result: the automorphism $\tau$ appearing in Proposition \ref{svsuperrig} intervenes, but so does a character $\Gamma \to Z_{G^{\mr{der}}}$ (see \cite[Theorem VIII.3.4.(a)]{margulis:discrete}), a further obstruction to analyzing the topological fundamental group representation associated to a companion. Second, even without this topological obstruction, again the uniqueness of descents in Lemma \ref{descentlem} (allowing Corollary \ref{svnearcompatible} to hold over $\kappa(v)$ and not just over $\CC$ or $\ol{\kappa(v)}$) no longer holds. In the Shimura variety setting we can establish compatibility of the canonical local systems upon projection to $G/G^{\mr{der}}$ (and to $G^{\mr{ad}}$, by Theorem \ref{mainthm}), but in summary we are left with two difficulties arising from $Z_{G^{\mr{der}}}$ in establishing compatibility in $G$. Nevertheless, in work-in-progress, we are optimistic of being able to combine our methods with the results of \cite{bakker-shankar-tsimerman:canonicalmodels} to resolve the general case.
\end{rmk}

\newcommand{\etalchar}[1]{$^{#1}$}
\def\cprime{$'$}
\providecommand{\bysame}{\leavevmode\hbox to3em{\hrulefill}\thinspace}
\providecommand{\MR}{\relax\ifhmode\unskip\space\fi MR }

\providecommand{\MRhref}[2]{%
  \href{http://www.ams.org/mathscinet-getitem?mr=#1}{#2}
}
\providecommand{\href}[2]{#2}


\begin{thebibliography}{BHKT19}

\bibitem[Art69]{artin:approximation}
M.~Artin, \emph{Algebraic approximation of structures over complete local
  rings}, Inst. Hautes \'{E}tudes Sci. Publ. Math. (1969), no.~36, 23--58.
  \MR{268188}

\bibitem[BST24]{bakker-shankar-tsimerman:canonicalmodels}
Benjamin Bakker, Ananth~N. Shankar, and Jacob Tsimerman.
\newblock Integral canonical models of exceptional shimura varieties.
\newblock {\em arXiv preprint arXiv:2405.12392v1}, 2024.


\bibitem[BCE{\etalchar{+}}18]{bcelmpp}
G.~{Boxer}, F.~{Calegari}, M.~{Emerton}, B.~{Levin}, K.~{Madapusi Pera}, and
  S.~{Patrikis}, \emph{{Compatible families of Galois representations
  associated to the exceptional group $E_6$}}, {to appear in Forum of Math.
  Sigma, arXiv:1805.09777} (2018).

\bibitem[BHKT19]{bhkt:fnfieldpotaut}
Gebhard B\"{o}ckle, Michael Harris, Chandrashekhar Khare, and Jack~A. Thorne,
  \emph{{$\hat G$}-local systems on smooth projective curves are potentially
  automorphic}, Acta Math. \textbf{223} (2019), no.~1, 1--111. \MR{4018263}

\bibitem[CK16]{cadoret-kret:galois-generic}
Anna Cadoret and Arno Kret, \emph{Galois-generic points on {S}himura
  varieties}, Algebra Number Theory \textbf{10} (2016), no.~9, 1893--1934.
  \MR{3576114}
  
\bibitem[CO12]{chai-oort:AVisogJac}
Ching-Li Chai and Frans Oort.
\newblock Abelian varieties isogenous to a {J}acobian.
\newblock {\em Ann. of Math. (2)}, 176(1):589--635, 2012.

\bibitem[Chi04]{chin:indl}
CheeWhye Chin, \emph{Independence of {$l$} of monodromy groups}, J. Amer. Math.
  Soc. \textbf{17} (2004), no.~3, 723--747. \MR{2053954}

\bibitem[Com19]{commelin:abelian}
Johan~M. Commelin, \emph{On compatibility of the {$\ell$}-adic realisations of
  an abelian motive}, Ann. Inst. Fourier (Grenoble) \textbf{69} (2019), no.~5,
  2089--2120. \MR{4018256}

\bibitem[Con14]{conrad:luminy}
Brian Conrad, \emph{Reductive group schemes}, Autour des sch\'{e}mas en
  groupes, Panoramas et Synth\`eses [Panoramas and Syntheses], vol. 42-43,
  Soci\'et\'e Math\'ematique de France, Paris, 2014, p.~458.

\bibitem[Del73]{deligne:constantes}
Pierre Deligne, \emph{Les constantes des {\'e}quations fonctionnelles des
  fonctions l}, Modular functions of one variable II, Springer, 1973,
  pp.~501--597.

\bibitem[Del79]{deligne:canonicalmodels}
\bysame, \emph{Vari\'et\'es de {S}himura: interpr\'etation modulaire, et
  techniques de construction de mod\`eles canoniques}, Automorphic forms,
  representations and {$L$}-functions ({P}roc. {S}ympos. {P}ure {M}ath.,
  {O}regon {S}tate {U}niv., {C}orvallis, {O}re., 1977), {P}art 2, Proc. Sympos.
  Pure Math., XXXIII, Amer. Math. Soc., Providence, R.I., 1979, pp.~247--289.
  \MR{546620 (81i:10032)}

\bibitem[Dri12]{drinfeld:deligneconj}
Vladimir Drinfeld, \emph{On a conjecture of {D}eligne}, Mosc. Math. J.
  \textbf{12} (2012), no.~3, 515--542, 668. \MR{3024821}

\bibitem[Dri18]{drinfeld:pross}
\bysame, \emph{On the pro-semisimple completion of the fundamental group of a
  smooth variety over a finite field}, Adv. Math. \textbf{327} (2018),
  708--788. \MR{3762002}

\bibitem[EG18]{esnault-groechenig:rigid}
H\'{e}l\`ene Esnault and Michael Groechenig, \emph{Cohomologically rigid local
  systems and integrality}, Selecta Math. (N.S.) \textbf{24} (2018), no.~5,
  4279--4292. \MR{3874695}

\bibitem[GH19]{getz-hahn:automorphicbook}
RJ~Getz and Heekyoung Hahn, \emph{An introduction to automorphic
  representations with a view toward trace formulae, gtm series}, 2019.

\bibitem[Gey78]{geyer:approx}
Wulf-Dieter Geyer.
\newblock Galois groups of intersections of local fields.
\newblock {\em Israel J. Math.}, 30(4):382--396, 1978.

\bibitem[GM06]{grothendieck:tame}
Alexandre Grothendieck and Jacob~P Murre, \emph{The tame fundamental group of a
  formal neighbourhood of a divisor with normal crossings on a scheme}, vol.
  208, Springer, 2006.

\bibitem[KS10]{kerz-schmidt:tameness}
Moritz Kerz and Alexander Schmidt, \emph{On different notions of tameness in
  arithmetic geometry}, Mathematische Annalen \textbf{346} (2010), no.~3, 641.

\bibitem[Kis17]{kisin:modpabelian}
Mark Kisin, \emph{{${\rm mod}\,p$} points on {S}himura varieties of abelian
  type}, J. Amer. Math. Soc. \textbf{30} (2017), no.~3, 819--914. \MR{3630089}

\bibitem[KPS15]{kisin-madapusipera-shin:hondatate}
Mark Kisin, K~Madapusi Pera, and Sug~Woo Shin, \emph{Honda-tate theory for
  shimura varieties}, preprint (2015).

\bibitem[KZ21]{kisin-zhou}
Mark Kisin and Rong Zhou, \emph{Independence of $\ell$ for frobenius conjugacy
  classes attached to abelian varieties}, 2021.

\bibitem[KP22]{klevdal-patrikis:G-rigid}
Christian Klevdal and Stefan Patrikis, \emph{{$G$}-cohomologically rigid local
  systems are integral}, Trans. Amer. Math. Soc. \textbf{375} (2022), no.~6,
  4153--4175. \MR{4419055}

\bibitem[KP23]{klevdal-patrikis:SVcompatiblearXiv}
Christian Klevdal and Stefan Patrikis.
\newblock Compatibility of canonical $\ell$-adic local systems on shimura
  varieties.
\newblock {\em arXiv preprint arXiv:2303.03863v1}, 2023.

\bibitem[Kot90]{kottwitz:annarbor}
Robert~E. Kottwitz.
\newblock Shimura varieties and {$\lambda$}-adic representations.
\newblock In {\em Automorphic forms, {S}himura varieties, and {$L$}-functions,
  {V}ol.\ {I} ({A}nn {A}rbor, {MI}, 1988)}, volume~10 of {\em Perspect. Math.},
  pages 161--209. Academic Press, Boston, MA, 1990.
  
\bibitem[Kot92a]{kottwitz:lambda-adic}
Robert~E. Kottwitz.
\newblock On the {$\lambda$}-adic representations associated to some simple
  {S}himura varieties.
\newblock {\em Invent. Math.}, 108(3):653--665, 1992.

\bibitem[Kot92b]{kottwitz:modppoints}
Robert~E. Kottwitz.
\newblock Points on some {S}himura varieties over finite fields.
\newblock {\em J. Amer. Math. Soc.}, 5(2):373--444, 1992.

\bibitem[LR87]{langlands-rapoport}
R.~P. Langlands and M.~Rapoport, \emph{Shimuravariet\"{a}ten und {G}erben}, J.
  Reine Angew. Math. \textbf{378} (1987), 113--220. \MR{895287}

\bibitem[LS13]{lan-suh:vanishingncpel}
Kai-Wen Lan and Junecue Suh, \emph{Vanishing theorems for torsion automorphic
  sheaves on general {PEL}-type {S}himura varieties}, Adv. Math. \textbf{242}
  (2013), 228--286. \MR{3055995}

\bibitem[LO10]{lieblich-olsson-pi1}
Max Lieblich and Martin Olsson, \emph{Generators and relations for the
  \'{e}tale fundamental group}, Pure Appl. Math. Q. \textbf{6} (2010), no.~1,
  Special Issue: In honor of John Tate. Part 2, 209--243. \MR{2591190}

\bibitem[LZ17]{liu-zhu:dRrigidity}
Ruochuan Liu and Xinwen Zhu, \emph{Rigidity and a {R}iemann-{H}ilbert
  correspondence for {$p$}-adic local systems}, Invent. Math. \textbf{207}
  (2017), no.~1, 291--343. \MR{3592758}

\bibitem[Litt]{litt:arithmetic2}
Daniel {Litt}.
\newblock {Arithmetic representations of fundamental groups II: finiteness}.
\newblock {\em arXiv e-prints}, page arXiv:1809.03524, Sep 2018.

\bibitem[Mar91]{margulis:discrete}
Gregori~A Margulis, \emph{Discrete subgroups of semisimple lie groups},
  vol.~17, Springer Science \& Business Media, 1991.

\bibitem[Mil83]{milne:canonical}
J.~S. Milne, \emph{The action of an automorphism of {${\bf C}$} on a {S}himura
  variety and its special points}, Arithmetic and geometry, {V}ol. {I}, Progr.
  Math., vol.~35, Birkh\"{a}user Boston, Boston, MA, 1983, pp.~239--265.
  \MR{717596}

\bibitem[Mil05]{milne:shimuraintro}
\bysame, \emph{Introduction to {S}himura varieties}, Harmonic analysis, the
  trace formula, and {S}himura varieties, Clay Math. Proc., vol.~4, Amer. Math.
  Soc., Providence, RI, 2005, pp.~265--378. \MR{2192012}

\bibitem[Pet20]{petrov:derham}
Alexander Petrov, \emph{Geometrically irreducible $ p $-adic local systems are
  de rham up to a twist}, arXiv preprint arXiv:2012.13372 (2020).

\bibitem[PR94]{platonov-rapinchuk}
Vladimir Platonov and Andrei Rapinchuk, \emph{Algebraic groups and number
  theory}, Pure and Applied Mathematics, vol. 139, Academic Press, Inc.,
  Boston, MA, 1994, Translated from the 1991 Russian original by Rachel Rowen.
  \MR{1278263}

\bibitem[PR01]{prasad-rapinchuk:irrtori}
Gopal Prasad and Andrei~S. Rapinchuk, \emph{Irreducible tori in semisimple
  groups}, Internat. Math. Res. Notices (2001), no.~23, 1229--1242.
  \MR{1866442}

\bibitem[PST{\etalchar{+}}21]{pila-shankar-tsimerman:andre-oort}
Jonathan Pila, Ananth~N Shankar, Jacob Tsimerman, H{\'e}l{\`e}ne Esnault, and
  Michael Groechenig, \emph{Canonical heights on shimura varieties and the
  andr$\backslash$'e-oort conjecture}, arXiv preprint arXiv:2109.08788 (2021).

\bibitem[Raj98]{rajan:sm1}
C.~S. Rajan, \emph{On strong multiplicity one for {$l$}-adic representations},
  Internat. Math. Res. Notices (1998), no.~3, 161--172. \MR{1606395
  (99c:11064)}

\bibitem[RR96]{rapoport-richartz}
M.~Rapoport and M.~Richartz.
\newblock On the classification and specialization of {$F$}-isocrystals with
  additional structure.
\newblock {\em Compositio Math.}, 103(2):153--181, 1996.

\bibitem[Ser97]{serre:mordell-weil}
Jean-Pierre Serre, \emph{Lectures on the {M}ordell-{W}eil theorem}, third ed.,
  Aspects of Mathematics, Friedr. Vieweg \& Sohn, Braunschweig, 1997,
  Translated from the French and edited by Martin Brown from notes by Michel
  Waldschmidt, With a foreword by Brown and Serre. \MR{1757192}

\bibitem[Ser98]{serre:ladic}
J-P. Serre, \emph{Abelian \protect{$\ell$}-adic representations and elliptic
  curves}, A\thinspace{}K Peters Ltd., Wellesley, MA, 1998, With the
  collaboration of Willem Kuyk and John Labute, Revised reprint of the 1968
  original.

\bibitem[sga71]{sga1}
\emph{Rev\^{e}tements \'{e}tales et groupe fondamental}, Lecture Notes in
  Mathematics, Vol. 224, Springer-Verlag, Berlin-New York, 1971, S\'{e}minaire
  de G\'{e}om\'{e}trie Alg\'{e}brique du Bois Marie 1960--1961 (SGA 1),
  Dirig\'{e} par Alexandre Grothendieck. Augment\'{e} de deux expos\'{e}s de M.
  Raynaud. \MR{0354651}

\bibitem[Spi99]{spivakovsky}
Mark Spivakovsky, \emph{A new proof of {D}. {P}opescu's theorem on smoothing of
  ring homomorphisms}, J. Amer. Math. Soc. \textbf{12} (1999), no.~2, 381--444.
  \MR{1647069}

\bibitem[{Sta}23]{stacks-project}
The {Stacks project authors}, \emph{The stacks project},
  \url{https://stacks.math.columbia.edu}, 2023.

\bibitem[UY13]{ullmo-yafaev:generalized}
Emmanuel Ullmo and Andrei Yafaev, \emph{Mumford-{T}ate and generalised
  {S}hafarevich conjectures}, Ann. Math. Qu\'{e}. \textbf{37} (2013), no.~2,
  255--284. \MR{3117743}

\bibitem[Vos98]{voskresenskii:birat}
V.~E. Voskresenski\u{\i}, \emph{Algebraic groups and their birational
  invariants}, Translations of Mathematical Monographs, vol. 179, American
  Mathematical Society, Providence, RI, 1998, Translated from the Russian
  manuscript by Boris Kunyavski [Boris \`E. Kunyavski\u{\i}]. \MR{1634406}

\bibitem[Wed99]{wedhorn:ordinary}
Torsten Wedhorn.
\newblock Ordinariness in good reductions of {S}himura varieties of {PEL}-type.
\newblock {\em Ann. Sci. \'Ecole Norm. Sup. (4)}, 32(5):575--618, 1999.
\bibitem[Yun14]{yun:exceptional}
Zhiwei Yun, \emph{Motives with exceptional {G}alois groups and the inverse
  {G}alois problem}, Invent. Math. \textbf{196} (2014), no.~2, 267--337.

\end{thebibliography}
\end{document}